\documentclass[aps, pre,preprint,superscriptaddress,onecolumn, 12pt]{article}

\usepackage{amssymb,amsmath,amsthm,amsfonts,amsbsy,mathrsfs, graphics,epsfig, graphicx, epstopdf,xcolor}
\usepackage[letterpaper,margin=2.35cm,nohead]{geometry}

\newcommand{\xb}{\bar{x}}
\newcommand{\PE}{\mathbb{E}}
\newcommand{\e}{\varepsilon}
\newcommand{\R}{\mathbb{R}}

\newtheorem{theorem}{Theorem}[section]
\newtheorem{corollary}[theorem]{Corollary}
\newtheorem{lemma}[theorem]{Lemma}

\newtheorem{remark}[theorem]{Remark}
\newcommand{\ep}{\varepsilon}

\usepackage{color}

\newcommand{\Red}{\textcolor{red}}

\newcommand{\pt}{\partial_{t}}
\newcommand{\px}{\partial_{x}}

\newcommand{\Dt}{\frac{\rm d}{{\rm d} t}}
\newcommand{\dx}{{\rm d} {x}}
\newcommand{\dy}{{\rm d} {y}}
\newcommand{\dt}{{\rm d} t }
\newcommand{\ds}{{\rm d} s }
\newcommand{\dd}{{\rm d}}
\newcommand{\dr}{{\rm d} r }
\newcommand{\dzeta}{{\rm d} \zeta }
\newcommand{\dz}{{\rm d} z }
\newcommand{\dxdt}{\dx \,\dt}
\newcommand{\lr}[1]{\left( #1 \right)}
\newcommand{\intO}[1]{\int_{\R} #1 \ \dx}

\newcommand{\intTO}[1]{\int_0^T\!\!\!\! \int_{\R} #1 \ \dxdt}
\newcommand{\eq}[1]{\begin{equation}
\begin{split}
#1
\end{split}
\end{equation}}
\newcommand{\eqh}[1]{\begin{equation*}
\begin{split}
#1
\end{split}
\end{equation*}}

\begin{document}

\title{On a PDE model for Learning in Stochastic Market Entry Games}

\author{
Esther Bou Dagher\textsuperscript{1}
\and
Misha Perepelitsa\textsuperscript{2}
\and
Ewelina Zatorska\textsuperscript{3}
}

\maketitle

\begin{center}
\textsuperscript{1}\emph{CEREMADE (CNRS UMR 7534), PSL University, Universit'e Paris-Dauphine}\
\emph{Place du Mar'echal de Lattre de Tassigny, 75775 Paris Cedex 16, France}\\
\texttt{esther.bou-dagher@dauphine.psl.eu}\\
\vspace{0.4em}
\textsuperscript{2}\emph{Department of Mathematics, University of Houston}\
\emph{Philip Guthrie Hoffman Hall, 3551 Cullen Blvd, Room 641, Houston, TX 77204-3008, USA}\\\
\texttt{maperepelitsa@uh.edu}\\
\vspace{0.4em}
\textsuperscript{3}\emph{Mathematics Institute, University of Warwick}\
\emph{Zeeman Building, Coventry CV4 7AL, United Kingdom}\\
\texttt{ewelina.zatorska@warwick.ac.uk}
\end{center}

\begin{abstract}

We study a continuum model for stochastic reinforcement learning in repeated market entry games. Starting from a discrete-time microscopic learning rule, we derive a Fokker--Planck-type equation for the distribution of agents’ propensities and, using a kinetic closure, obtain a nonlinear one-particle equation of a mean-field type.
For the resulting Cauchy problem, we prove existence and uniqueness of solutions and analyze their long-time behavior. The PDE captures two key phenomena observed in market entry dynamics: aggregate learning (the average number of entrants approaches market capacity) and sorting (propensities concentrate near extreme behaviors). The model also yields explicit characteristic time scales, showing that aggregate learning occurs faster than sorting, in agreement with experimental and computational evidence.

\end{abstract}

\medskip


\section{Introduction}
\label{Intro}
A class of games in the study of social and economic behavior, called market entry games, describes a conflict situation in which players (agents) choose between two strategies: enter the market or stay out. The agent’s payoff is determined solely by the number of agents who decide to enter. A well-known example of this class of games is the El Farol Bar game, introduced by Arthur \cite{Arthur1994}.

A typical game of this type has a single symmetric mixed equilibrium, as well as pure-strategy and mixed asymmetric equilibria.

Naturally, one is interested in whether any of the equilibrium strategies emerge when the game is played repeatedly by agents who, independently of each other, attempt to adapt to changing “market conditions”.

Reinforcement learning is a fundamental modeling approach for market adaptation, and more generally serves as a fundamental paradigm for animal and human behavior, see for example \cite{HARLEY1981} and references therein. To describe reinforcement learning processes, one introduces individual propensities of agents to play particular strategies. These propensities are updated after each round of the game and are determined by the agent’s payoffs, as in the basic reinforcement learning model introduced by Roth and Erev \cite{ROTH1995164}, or in the related variant, fictitious stochastic play, introduced by Fudenberg and Levine \cite{fudenberg1998}. Experimental studies of such models are reported in \cite{ROTH1995164, erev1998, erev1998b}.
The following patterns of behavior are typically observed in experimental studies of learning models in the long run \cite{duffy2005}:
a) the average number of entries per round quickly approaches the market capacity; this pattern is referred to as “aggregate learning”;
b) over a long run of repeated plays, players’ strategies converge to a pure-strategy equilibrium compatible with the market capacity; this pattern is called “sorting”.

Both phenomena are ubiquitous in market entry games in which agents use either basic reinforcement learning or fictitious stochastic play. Experimental studies show that aggregate learning emerges relatively quickly, whereas sorting takes a much longer time to occur; see Duffy and Hopkins \cite{duffy2005}.

The analytic treatment of learning models was carried out within the framework of stochastic approximation developed by Benaïm \cite{benaim1999}, and further applied by Duffy and Hopkins \cite{duffy2005} and Whitehead \cite{whitehead2008}. The authors prove that, in repeated market entry games with either basic reinforcement learning or fictitious stochastic play, and under fairly general conditions, agents’ strategies converge with probability one to a pure-strategy equilibrium. In that approach, learning models are viewed as stochastic dynamical systems of size $M$ (the number of agents), describing the individual propensities of all agents under a stochastic updating rule. The stochastic dynamical system is approximated by a deterministic system of ODEs whose steady states correspond to Nash equilibria, with pure-strategy equilibria being stable and mixed equilibria unstable.

Another type of learning in games and the corresponding representation of such processes through kinetic PDEs were considered in \cite{Perepelitsa2019, Perepelitsa2021}, where players are assumed to interact by playing a symmetric $2\times 2$ game in randomly matched pairs. The learning is implemented either by reinforcing the strategy of a player according to the accumulated payoff or by increasing/decreasing the number of players using the strategy.

The purpose of this paper is twofold. First, we consider a Fokker–Planck equation for the function $W$, which describes the distribution of the propensities of all agents. This equation is derived as an approximation of the distribution generated by the discrete-time stochastic learning process. Then, following a kinetic closure approach based on the independence hypothesis, we reduce the dimensionality of the model and obtain a kinetic equation for the distribution of the propensity of a randomly selected agent (the one-particle distribution function).
The resulting equation, \eqref{eq:f20}, is a non-linear transport-diffusion equation in which both the transport and diffusion coefficients are given by certain moments of the kinetic function.
This equation is a mean-field type equation where the transport velocity is defined by an expected (average) payoff. 

We prove an existence and uniqueness theorem for the Cauchy problem for equation \eqref{eq:f20} and show that the asymptotic behavior of solutions as $t \to +\infty$ corresponds to a state in which the average number of entrants is close to the market capacity, see the next section for details. That is, it corresponds to aggregate learning. Moreover, agents’ propensities become concentrated near the extreme values, corresponding to sorting. Thus, the one-particle, kinetic PDE effectively captures the asymptotic behavior of the stochastic learning process.

Second, the kinetic PDE approach provides formulas for the characteristic time scales of aggregate learning and sorting in terms of the microscopic parameters of the model.  Using explicit expressions for the drift and diffusion coefficients in the partial differential equation, we find that the time scale of aggregate learning is shorter than that of sorting, in agreement with the experimental studies mentioned above.

Kinetic equation \eqref{eq:f20} is similar to the kinetic equation (3.1) obtained in \cite{Degond2014} which was derived as a mean-field model for the behavior of rational agents interacting in a game-theoretic framework. In their model, the players change strategies in the direction of the gradient of the expected payoff with an exogenously added noise term. Conceptually, this is close to the reinforcement learning model considered in this paper. The difference is that in our model the players base their behavior on instantaneous payoffs rather than expected payoffs. Thus, the diffusion in our equation reflects the randomness of the actual actions taken by the players and depends functionally on the strategies currently selected by the players, rather than being due to a constant-variance white noise added to the mean field, as in \cite{Degond2014}.

%This difference is also reflected in the expressions for the transport coefficients. Both are of the mean-field type, but their formulas differ. Lastly, the mean-field model in \cite{Degond2014} is obtained in the limit of infinite population size $M\to+\infty.$ In such a case to get a non-trivial dynamics, one has to assume the proportion of the market capacity to population size, $M_c/M,$ remains strictly positive in the limit, which has slightly different economic meaning. Moreover, the sorting phenomena is a finite population effect and it is lost in such a limit.

% In many cases of interacting particle systems, the long-time behavior of solutions of kinetic equations can be characterized through a variational structure, from which the convergence to an equilibrium state can be deduced in a straightforward way. Some examples of such systems can be found 
% \cite{constantin2004, frouvelle2012, Degond2014}\blue{Ewelina more refs}.

In many cases of interacting particle systems, the long-time behavior of solutions of kinetic equations can be characterized through a variational structure, from which convergence to equilibrium can be derived rather directly; see, for example \cite{constantin2004, Carrillo, frouvelle2012, Degond2014, CaChZa, BaDeZa} and \cite{Otto2001} for a general variational structure/gradient flow viewpoint behind long-time asymptotics.

The equation we study, however, does not possess a natural “free energy” or Lyapunov functional that is minimized along trajectories and whose minimizers correspond to equilibrium states of learning and sorting. More precisely, there exists a non-decreasing ``energy'' functional, but its minima do not necessarily coincide with equilibrium states.

Moreover, the moments of the kinetic distribution that describe learning and sorting do not satisfy a closed system of ODEs. As a result, the dynamics cannot be reduced to a finite-dimensional system governing these aggregate quantities.

A natural starting point for analyzing the asymptotic behavior is the ``energy" inequality, \eqref{eq:energy}, which holds for generic solutions of the kinetic equation. It states that a weighted $L^2$ norm of the solution decreases in time, with the rate given by the $L^2$ norm of the its gradient.
The classical approach would be to estimate the rate using the Nash inequality and evaluate the rate of decay as a function of time and conclude that  the $L^2$ converges to zero. In our model, however, the diffusion coefficient is a certain moment  in $x$--variable of the solution itself. Its time dependence is not known known a-priori, and this direct approach does not produce a conclusive result.
What we found is that instead of the $L^2$ norm, we can consider the product of the weighted $L^2$ norm and a moment of the solution:
\[
\phi(t){}={}\beta(t)\left(\intO{p^2(x) f^2(x,t)}\right)^3,\quad \beta(t)=\intO{p(x)(1-p(x))f(x,t)},
\]
where function $\beta(t)$ appears as a lower bound  for the diffusion coefficient, see Section \ref{sec:ltb} for details. Using the information contained in the energy inequality, one can prove  that the function $\phi(t)$
converges to zero as $t\to+\infty$, which is sufficient to show that the mass of $f(\cdot,t)$ moves to the extreme values, that is, that the sorting phenomenon takes place.

 To show that the aggregate learning takes place we need to show that the average proportion of agents entering the market, expressed by the moment $\intO{p(x)f(x,t)}$ lies in the ``optimal" interval $((M_c-1)/M, M_c/M),$ where $M_c$ is the critical market capacity and $M$ is the total number of agents. We prove that, for large time, this moment remains close, within the bounds controlled by the model parameters, to a certain value $\kappa$ in the optimal interval. Hence, the proportion of agents entering the market approaches this value.

The proof is indirect and is based on the fact that the transport coefficient in the equation is proportional to the quantity
\begin{equation}
\label{intro:a(t)}
a(t)=\intO{p(x)f(x,t)} -\kappa
\end{equation}
measuring the deviation of the proportion of agents entering the market from $\kappa.$ We show first, that either $a(t)$ is sufficiently small or it has a limit at $t=+\infty.$ Then, we use a contradiction-type argument to prove that if this limit is positive (negative), then the mass of $f(\cdot,t)$ moves to the right (left) extreme value of $x,$ in which case $a(t)$ must be negative (positive).

Proving this fact, requires a careful study of the balance between the transport and the diffusion parts of the equation. For this reason, we  restrict the model parameters, see Theorem \ref{thm:long time}, to a regime in which transport ``dominates" diffusion. At the same time we need to restrict the probability function $p(x),$
appearing in \eqref{intro:a(t)} to have a positive, but arbitrarily small, lower bound. With these assumptions, we use the kinetic equation to evaluate the moment  $
\intO{f(x,t)\psi(x,t)}$ where $\psi(x,t)$ is a suitable ``test" function solving a transport PDE,  with the transport part as in the kinetic equation, and with the initial data growing exponentially at $x=\pm \infty.$  Using the pointwise estimates on function $\psi$ and its derivatives, we estimate the effect 
of diffusion on the moment $\intO{f(x,t)\psi(x,t)}$ as being at most exponential in time. Comparing this to the values of function $\psi(x,t)$ which also grows exponentially but at a strictly faster rate,
we conclude that the mass of $f(\cdot,t)$ to the left (right) of an arbitrary point $x_0$ decreases to zero as time increases. Thus, this yields the desired conclusion.

The paper is organized as follows.

In Section \ref{sec:game}, we introduce the formal definition of the market entry game. In Section \ref{sec:pde}, we derive the corresponding kinetic equation. In Section \ref{sec:existence}, we establish the existence and uniqueness of strong solutions to the associated Cauchy problem.
We start the section by deriving several a priori estimates. These estimates are also used to prove uniqueness of strong solutions and, in Section \ref{sec:ltb}, to analyze the long-time behavior of solutions.
The strategy follows a classical approach: we begin by regularizing the potentially degenerate diffusion coefficient by a strictly positive one and then linearize the equation and apply standard existence results for linear parabolic equations.
Next, using a fixed-point argument, we construct solutions to the nonlinear problem with regularized diffusion. Finally, we pass to the limit in the regularization parameter and recover a solution to the original problem.

Section \ref{sec:ltb} is devoted to the study of long-time behaviour of our solution.  We establish that, asymptotically, the solution of the kinetic PDE captures the state of aggregate learning and sorting, the latter being considerably more delicate.

%%%%%%%%%%%%%%%%%%%%%%%%%%%%%%%%%%%%%%%%%%%%%%%%%%%%%%%%

\section{Reinforcement learning process}
\label{sec:game}
We follow Erev and Rapoport \cite{erev1998} for the exposition of the Market Entry Game. There are $M$ agents participating in the game. Let $\delta^i$ denote the indicator function for agent $i$ to enter the game: $\delta^i=1$ if the agent enters, and $\delta^i=0,$ otherwise. Let $M_c\geq1$ be the critical capacity of the market, which separates 
 the under-populated market (positive effect) from the over-populated market
 (negative effect). Let $m$ be the number of agents who enter the market, $h$ be the payoff per unit of difference of $m$ from the critical capacity $M_c,$ and $v>0$ be the compensation for participating in the game. The payoff to agent $i$ to defined to be equal to 
\begin{equation}
\label{pi}
\pi^i{}={}\left\{
\begin{array}{ll}
v & \mbox{if } \delta^i=0,\\
v+h(M_c-m) & \mbox{if } \delta^i=1.
\end{array}
\right.
\end{equation} 
This game has a number of pure and mixed strategy Nash equilibria. The following description is taken from \cite{duffy2005}. If $M_c,$ which we shall assume in this paper, is an integer, any profile of pure strategies which is consistent with either $M_c$ or $M_c-1$ entrants is a Nash equilibrium. 

Moreover, there is a continuum of equilibria in which $M_c-1$ players enter, $M-M_c$ stay out, and one player enters with any probability.

Additionally, for $M_c>1,$ there is a symmetric Nash equilibrium. This has the form
\[
\bar{p}^i=\frac{M_c-1}{M-1},\quad i=1,\ldots,M,
\]
where $\bar{p}^i$ is the probability to enter for player $i$. Note that the expected number of entrants in the symmetric mixed equilibrium is 
$M(M_c-1)/(M-1)\in (M_c-1, M_c).$ There are additional mixed asymmetric mixed equlibria of the form $j<M_c-1$ players enter with probability one, $k<M-M_c$ players stay out and the remaining $M-j-k$ players enter with probability $(M_c-1-j)/(M-j-k-1).$ In this case, the expected number of entrants is in the range $(M_c-1, M_c).$ 
In all the equilibria mentioned above the expected number of entrants is between $M_c-1$
 and $M_c.$

In the reinforcement learning process,  the game is played repeatedly, and the state of the agent $i,$ is defined by the propensities to enter and stay out after $n^{th}$ round of the game: 
\[
(X^i_{1,n},X^i_{2,n})\in\mathbb{R}^2.
\]
The propensities change over one round of the game by the amount of payoff:
\[
X^i_{1,n+1} = X^i_{1,n}+h(v+M_c-m_n),\quad
X^i_{2,n+1} = X^i_{2,n},
\]
if agent $i$ enters the market, and
\[
X^i_{1,n+1} = X^i_{1,n},\quad X^i_{2,n+1} = X^i_{2,n} + v,
\]
if agent $i$ doesn't enter. Here, $m_n$ is the number of agents who entered the market at round $n+1.$
Furthermore, the model specifies the probability that agent $i$ enters the market at each round as a function of propensity. For example, the probability used in \cite{duffy2005},
\[
\frac{X^i_{1,n}}{X^i_{1,n}+X^i_{2,n}}.
\]
To simplify the presentation, we assume that in \eqref{pi} and other formulas, $v=0.$ In this case, the propensity to stay out does not change in time: $X^i_{2,n}{}={}X^i_{2,0}.$ Furthermore, we assume that initially, $X^i_{2,0}$  are all equal to some $x_0.$ Consequently, we only need to consider one propensity, the propensity to enter the market, $X^i_{1,n}$ which we denote by
$
X^i_n.
$
Under such assumptions, all agents use the same function to map propensities to probabilities. Given propensity $X^i_n=x$ of agent $i,$ the probability to enter for this player is 
\[
p(x) =\frac{x}{x+x_0},
\]
where it is assumed that propensities cannot take negative values. 
In fact, the explicit formula for the probability function will not be needed in our analysis, and we opt to use a generic monotonely increasing function $p=p(x)$ with values in the interval $(0,1]$  defined on all of $x\in\mathbb{R},$ eliminating by this the non-negativity restriction on propensities.

In this notation,   the probability to enter for agent $i$ equals $p(X^i_n).$ 

Summarizing the model, the state of agent $i$ is described by propensity $X^i_n,$ and the rule for its update is given as
\begin{equation*}
\tag{M1}
X^i_{n+1}{}={}X^i_n{}+{}h(M_c-m_n)\delta^i_n,
\end{equation*}
where $\delta^i_n$ is the indicator function of the action of player $i$ in $n^{th}$ game, 
% \red{Why is there a mismatch between index $n$ and $n+1$ here? How is $X_{n+1}$ updated using already something that happened in round $n+1$?}
% \blue{Fixed}, 
and $m_n$ is the number of agent who enter the game: $m_n = \sum_{i=1}^M\delta^i_n$ where $\delta^i_n$ are independent random variables such that $\mbox{\rm Prob}(\delta^i_n=1){}={} 1- \mbox{\rm Prob}(\delta^i_n = 0) = p(X^i_n).$
It will be convenient for analysis to use a time variable $t$ with arbitrary units and time step $\tau$ and assume that rounds of the game occur at intervals $t_n=\tau n,\,n\in\mathbb{N}.$ This will provide us with the a control parameter $\frac{1}{\tau},$ the number of rounds per unit of time.

%\red{For fixed $h>0$, propensities evolve on the lattice $h \mathbb{Z}$,  in the kinetic limit $h\to0$ we obtain a continuous density.}
%This means that the vector of random variables 
%\[
%X_n{}={}(X_n^1,...,X_n^M)
%\]
%is a Markov chain on the state space $\mathbb{S}^M,$ where $\mathbb{S}$ is a mesh of discrete propensities $\{kh\}_{k\in\mathbb{Z}}.$

We say that the system moves, on average,  to a state of aggregate learning if the expected number of agents entering the market approaches the interval of values in the Nash equilibria, $[M_c-1,M_c]$:
\begin{equation}
\label{def:learning}
\liminf_{n\to\infty} \mathbb{E}\left[m_n\right] \geq M_c-1,\quad \limsup_{n\to\infty} \mathbb{E}\left[m_n\right] \leq M_c.
\end{equation}

The sorting of agents in the systems is described by the property that asymptotically agents are separated into groups with propensities equal to either $+\infty,$ or $-\infty:$ $\forall R>0,\,i=1..M,$
\begin{equation}
\label{def:sorting}
\lim_{n\to\infty} \mathbb{E}\left[\{ X^i_n\in (-R,R)\} \right] {}={}0.
\end{equation}

%%%%%%%%%%%%%%%%%%%%%%%%%%%%%%%%%%%%%%%%%%%%%%%%%%%%%%%%%%%%%%

\section{Kinetic PDEs for the reinforcement learning}
\label{sec:pde}

Denote by 
\[
W(\bar{x},t_n),\quad \bar{x}{}={}(x^1,...,x^M),
\]
the probability density function of the random vector $X_n.$  
The change of $W$ from time $t_{n}$ to $t_{n+1}$ is described by 
%the change in the moment with a smooth test function $\phi(\xb):$
%\begin{multline}
%\label{eq:one step}
%\int \phi(\xb)W(\xb,t_{n+1})\,d\xb {}={}\PE\left[\phi(X_{n+1})\right]{}=%{}\PE\left[\PE\left[\phi(X_{n+1})\,|\, X_n\right]\right]\\
%{}={}\int \PE\left[\phi(X_{n+1})\,|\,X_n=\bar{x} \right]W(\bar{x},t_n)\,d\xb\\{}={}
%\int \PE\left[\phi(\xb+Mh(\kappa-m_n/M)\bar{\delta}_n)\,|\,X_n=\bar{x} %\right]W(\bar{x},t_n)\,d\xb.
%\end{multline}
%In the last formula
%\[
%m_n=\sum_{j=1}^M \delta^j_n,
%\]
%$\bar{\delta}_n=(\delta^1_n,..,\delta^M_n)$ is the vector of characteristic functions of random events that agent $1$ through $M$ enter the market at time $t_n.$ 
%This equation is derived using the following Kramers-Moyal expansion of the density function $W(\xb,t),$ see Risken\cite{risken1996}, Ch.4.
the Kolmogorov equation for the increment of $W$ over a time step $\tau$ can be used to obtain an asymptotic expansion of $W$ in powers of $h:$
    \begin{equation*}
    W(\xb,t_{n+1})-W(\xb,t_n) {}={}
    -h\sum_i \partial_{x^i}(A_i(\xb)W) + \frac{h^2}{2}\sum_{i,k} \partial^2_{x^i x^k} (D_{ik}(\xb) W){}+{}O(h^3),
    %\sum_{\gamma\in \mathbb{N}^M: |\alpha|\geq1}
    %\frac{(-1)^{|\alpha|}(Mh)^{|\alpha|}}{\alpha!} \partial^\alpha_{\xb} (N_\alpha(\xb)W(\xb,t_n)),
    \end{equation*}
where
\begin{eqnarray*}
A_i(\xb) &=& \PE\left[ (M_c - m)\delta^i\,|\, X=\xb\right],\\
\label{def:Dik}
D_{ik}(\xb) &=& \PE\left[ (M_c - m)^2\delta^i\delta^k\,|\, X=\xb\right],\quad i,k=1..M.
\end{eqnarray*}
Here, $\{\delta^i\}_{i=1}^M$
is a mutually independent set of random variables with values in $\{0,1\},$ with $P(\delta^i=1)=p(x^i),$ $i=1..M,$ and  $m=\sum_{i=1}^M\delta^i,$
see, for example, in Risken\cite{risken1996}, Ch.4.   
The first order approximation coefficients $hA_i,\,i=1..M,$ are precisely the expected payoffs to players over one round of the play. Thus, at the first order of approximation the system is driven by the mean field dynamics.

To obtain a time-continuous equation for $W,$ we
 assume that $W(\xb,t_n)$ is the restriction at $t=t_n$ of a smooth density function $W(\xb,t)$ defined for all times.
Under this assumption we derive the following kinetic equation for $W:$
\begin{equation}
\label{eq:FK0_approx}
    \partial_t W + \frac{h}{\tau}\sum_i \partial_{x^i}(A_i(\xb)W) - \frac{h^2}{2\tau}\sum_{i,k} \partial^2_{x^i x^k} (D_{ik}(\xb) W){}={}\frac{O(h^3)}{\tau}.
\end{equation}

It can be verified that the  matrix $D,$ with entries $\{D_{ik}\},$ is a symmetric and positive semi-definite matrix:
\[
z^TDz{}={}\mathbb{E}\left[(M_c-m)^2(z\cdot \delta)^2\,|\, X=\bar{x}\right]{}\geq{}0,\quad \forall z\in\mathbb{R}^M.
\]
Moreover, $D$ is positive definite unless the system is in the state where the number of market entrants $m$ equals the critical capacity $M_c.$
Thus, the third term in \eqref{eq:FK0_approx} represents diffusion, which may be degenerate.

%The diffusion term captures the effect of randomness: it spreads the density over the state space.
%In contrast, the transport term reflects a persistent bias toward an equilibrium steady state, through the transport coefficients $A_i.$

%This drift–diffusion structure is characteristic of many kinetic equations, including the Fokker–Planck equation for the Ornstein–Uhlenbeck process \cite{risken1996}.

We will approximate \eqref{eq:FK0_approx} by keeping terms of order 
$h/\tau$ and 
$h^2/\tau,$ as these govern the leading-order dynamics of the underlying time-discrete stochastic process, and truncate higher-order terms. That is, we assume that
\begin{equation}
\label{eq:h_tau}
h,\tau \ll 1,\quad \frac{h^2}{\tau} \sim 1.
\end{equation}
Without loss of generality, we will also set 
\[
\frac{h^2}{\tau}=1.
\]
Thus, we arrive at equation
\begin{equation}
\label{eq:FK0}
    \partial_t W + \frac{1}{\sqrt{\tau}}\sum_i \partial_{x^i}(A_i(\xb)W) - \frac{1}{2}\sum_{i,k} \partial^2_{x^i x^k} (D_{ik}(\xb) W){}={}0.
\end{equation}

The results of this paper indicate that this truncation provides a good approximation: the asymptotic behavior of solutions to the truncated PDE agrees with that of the original stochastic process.

%The first-order part of \eqref{eq:FK0} corresponds to transport of 
%$W$ along the flow generated by the ODE system in the space of propensities:
%\[
%\frac{dx^i}{dt}{}={}\frac{h}{\tau}A_i(\xb){}={}\frac{h}{\tau}p(x^i)\left(M_c - 1-\sum_{j\not=i} p(x^j)\right), \quad i=1..M.
%\]
%Equivalently, in the space of probabilities (strategies) $p(x^i),$ we obtain
%\begin{equation}
%\label{eq:ODEs}
%\frac{dp(x^i)}{dt}{}={}\frac{h}{\tau}p(x^i)p'(x^i)\left(M_c - 1-\sum_{j\not=i} p(x^j)\right), \quad i=1..M.
%\end{equation}
%This system is an analog of the discrete-time equation (6) from \cite{duffy2005}. As shown there, the dynamics drives the agents' strategies toward steady states corresponding to pure or mixed Nash equilibria.  In particular, the expected number of agents entering the market converges to the range $(M_c-1,M_c].$

In the next step, we reduce the dimension of the problem  by deriving a kinetic equation for the probability density of the propensity of a randomly selected agent (the one-particle distribution function), using an independence hypothesis.

Let $f$ be the one-particle distribution function
\begin{equation}
\label{def:f}
f(x,t)  = \frac{1}{M}\sum_{j=1}^M \int_{\R^{M-1}}W(\xb)\big|_{x^j=x} \dd\xb^j,\quad x\in\mathbb{R},
\end{equation}
where $\xb^j=(x^1,..,x^{j-1},x^{j+1},..,x^M).$
Two- and three-particle distribution functions are defined, respectively, as
\begin{equation*}
%\label{def:g}
g(x,y,t)  = \frac{1}{M(M-1)}\sum_{j,k:j\not=k} \int_{\R^{M-2}} W(\xb)\big|_{\substack{x^j=x\\x^k=y}} \dd\xb^{jk},\quad x\in\mathbb{R},
\end{equation*}
\begin{equation*}
%\label{def:h}
l(x,y,z,t)  = \frac{1}{M(M-1)(M-2)}\sum_{i,j,k:i\not=j\not=k} \int_{\R^{M-3}} W(\xb)\big|_{\substack{x^i=x\\x^j=y\\x^k=z}} \dd\xb^{ijk},\quad x\in\mathbb{R},
\end{equation*}
where $\xb^{jk}$ is the vector of all variables, except $x^j$ and $x^k,$ and $\xb^{ijk}$ is the vector of all variables, except $x^i,$ $x^j$ and $x^k.$

To determine the equation for $f,$ we apply the averaging operator from \eqref{def:f} to equation \eqref{eq:FK0}.

One computes
\eq{
\label{eq:L1_1}
 \frac{1}{M}\sum_j \int_{\R^{M-1}} \sum_i\partial_{x^i}(A_i(\xb)W)\big|_{x^j=x}\,\dd\xb^j
 {}&={} \frac{1}{M}\sum_j \int_{\R^{M-1}} \partial_{x^j}(A_j(\xb)W)\big|_{x^j=x}\,\dd\xb^j\\
 {}&={}\partial_x \frac{1}{M}\sum_j \int_{\R^{M-1}} A_j(\xb)W\big|_{x^j=x}\,\dd\xb^j
}
and
\eq{
\label{eq:L1_2}
    &\frac{1}{M}\sum_j \int_{\R^{M-1}} A_j(\xb)W\big|_{x^j=x}\,\dd\xb^j{}\\
    &\qquad={}\frac{1}{M}\sum_j \int_{\R^{M-1}} \PE\left[(M_c-m)\delta^j\,|\,X = \xb \right]W\big|_{x^j=x}\,\dd\xb^j\\
    {}&\qquad={}\frac{1}{M}\sum_j \int_{\R^{M-1}} \PE\left[
    M_c\delta^j - \delta^j -\sum_{i:i\not=j}\delta^i\delta^j\,|\,X=\xb
\right]W\big|_{x^j=x}\,\dd\xb^j\\
    {}&\qquad={}(M_c -1)p(x)f(x,t)
    {}-{}\frac{1}{M}\sum_{i\not=j}
    \int_{\R^{M-1}} p(x^i)p(x^j)W\big|_{x^j=x}\,\dd\xb^j\\
    {}&\qquad={}(M_c -1)p(x)f(x,t)
    {}-{}(M-1)p(x) \int_{\R} p(y)g(x,y,t)\,\dy.
}
To find the kinetic closure for function $f$, we will use the ``molecular chaos" hypotheses
\begin{equation}
\label{cond:chaos}
g(x,y,t)=f(x,t)f(y,t),\quad l(x,y,z,t)=f(x,t)f(y,t)f(z,t),
\end{equation}
expressing the independence of the propensities to play of two or three randomly selected players. 

Beyond the game-theoretic and stochastic-approximation literature, our continuum closure is naturally related to the theory of weakly interacting particle systems and mean-field limits \cite{Dobrushin, McKean, Sznitman}.
For a broad overview of models and techniques, we refer to the surveys on the topic in \cite{Meleard,ChaintronDiez}. In the present work, we do not prove a propagation-of-chaos result for the underlying learning dynamics; rather, we use a mean-field closure in this spirit to derive and analyze a Fokker–Planck-type PDE for the distribution of propensities.

In fact, only the second condition suffices, because the first follows from that. The independence condition is a plausible assumption to make, when the number of players is large, since the players interact with each other only through the mean attendance $m_n.$

Using this hypothesis in \eqref{eq:L1_2},
and defining
\[
\kappa{}={}\frac{M_c-1}{M-1}\in[0,1],
\]
we find that 
\eqh{
    &\frac{1}{M}\sum_j \int_{\R^{M-1}} A_j(\xb)W\big|_{x^j=x}\,\dd\xb^j{}\\
    &\qquad={}
    \left(M_c-1 -(M-1)\intO{p(x) f(x,t)}\right)p(x)f(x,t)\\
    {}&\qquad={}(M-1)\left(\intO{(\kappa - p(x))f(x,t)}\right) p(x)f(x,t)=(M-1)a(t)p(x)f(x,t),
}
    where
    \begin{equation*}
        %\label{def:a0}
      a(t) = \intO{(\kappa - p(x))f(x,t)}.
    \end{equation*}

Substituting this back in  \eqref{eq:L1_1}, gets us
\begin{equation*}
\frac{1}{M}\sum_j \int_{\R^{M-1}}\sum_i\partial_{x^i}(A_i(\xb)W)\big|_{x^j=x}\,\dd\xb^j
 {}={} (M-1)a(t)\partial_x(p(x)f(x,t)).
 \end{equation*}

The computation for the diffusion term is similar.
\eqh{
 \frac{1}{M}\sum_j \int_{\R^{M-1}} \sum_{i,k}\partial^2_{x^ix^k}(D_{ik}(\xb)W)\big|_{x^j=x}\,\dd\xb^j
 {}&={} \frac{1}{M}\sum_j \int_{\R^{M-1}} \partial^2_{x^jx^j}(D_{jj}(\xb)W)\big|_{x^j=x}\,\dd\xb^j\\
 {}&={}\partial^2_{xx} \frac{1}{M}\sum_j \int_{\R^{M-1}} D_{jj}(\xb)W\big|_{x^j=x}\,\dd\xb^j.
}
We can write
\eqh{
   &\frac{1}{M}\sum_j \int_{\R^{M-1}} D_{jj}(\xb)W\big|_{x^j=x}\,\dd\xb^j{}\\
   &={}\frac{1}{M}\sum_j\int_{\R^{M-1}} \PE\left[(M_c-m)^2(\delta^j)^2\,|\,X = \xb \right]W\big|_{x^j=x}\,\dd\xb^j\\
   {}&={}\frac{1}{M}\sum_j\int_{\R^{M-1}} \PE\left[\left(M_c^2-2M_c\sum_i\delta^i +\sum_{i,k}\delta^i\delta^k\right)\delta^j\,|\,X = \xb \right]W\big|_{x^j=x}\,\dd\xb^j\\
   {}&={}\frac{1}{M}\sum_j\int_{\R^{M-1}} \left(M_c^2p(x^j)
   -2M_cp(x^j) -2M_c\sum_{i:i\not=j}p(x^i)p(x^j) +p(x^j) \right.
   \\
   &\qquad\qquad\qquad\left.+{3}p(x^j)\sum_{i:i\not=j}p(x^i)+\sum_{i,k:i\not=k\not=j}p(x^i)p(x^k)p(x^j)\right)W\big|_{x^j=x}\,\dd\xb^j\\
{}&={}(M_c^2-2M_c+1)p(x)f(x,t)
{}+{}({3}-2M_c)p(x)(M-1)\int_{\R}p(y)g(x,y,t)\,\dy\\
{}&\qquad\qquad\qquad+{}p(x)(M-1)(M-2)\int_{\R^2} p(y)p(z){l}(x,y,z,t)\,\dy\,\dd z.
}
Note that, the factor $3$ comes from the three ways of matching indices in
$\sum_{i,k}\delta^i\delta^k\delta^j$: namely $i=j$, $k=j$, or $i=k$ (with the remaining index different from $j$).

% \textcolor{red}{Question: The coefficient '3' is for the cases a) $j=i,j\neq k,i\neq k$  b) $j=k, j\neq i ,k\neq i$  and c) $i=k, j\neq i, j\neq k$ ?} 
% \textcolor{blue}{Correct}

\vskip 10pt

Using the closure relations \eqref{cond:chaos}, we obtain from the last equality:
\eqh{
&\frac{1}{M}\sum_j \int_{\R^{M-1}} D_{jj}(\xb)W\big|_{x^j=x}\,\dd\xb^j{}\\
&\quad={}
\left((M_c-1)^2  + (3-2M_c)(M-1)\intO{p(x)f(x,t)}\right. \\
&{}\qquad\qquad\qquad\qquad\left.+{}(M-1)(M-2)\left(\intO{p(x)f(x,t)}\right)^2\right)p(x)f(x,t)\\
{}&\quad={}\left[(M-1)^2(\kappa -\intO{pf})^2 + (M-1)\intO{pf}\intO{(1-p)f}\right]p(x)f(x,t)\\
{}&\quad={}\left[(M-1)^2a(t)^2+ (M-1)\intO{pf}\intO{(1-p)f}\right]p(x)f(x,t). 
}
With that, the diffusion term becomes
\[
\frac{1}{M}\sum_j \int_{\R^{M-1}}\sum_{i,k}\partial^2_{x^ix^k}(D_{ik}(\xb)W)\big|_{x^j=x}\,\dd\xb^j{}={}
\left((M-1)^2a^2(t)+(M-1)b(t)\right)\partial^2_{xx}(p(x)f(x,t)),
\]
where 
\begin{equation*}
% \label{def:b0}
b(t) = \intO{p(x)f(x,t)}\intO{(1-p(x))f(x,t)}.
\end{equation*}

Then, the equation for  the kinetic density $f$ is
\begin{equation}
\label{eq:f20}
\partial_t f + \frac{(M-1)}{\sqrt{\tau}}a(t)\partial_x(pf) - \frac{(M-1)^2}{2}\left(a^2(t)+\frac{1}{M-1}b(t)\right)\partial^2_{xx}(pf){}={}0,
\end{equation}
for $(x,t)\in\mathbb{R}\times\mathbb{R}^+$.

The aggregate learning and the sorting conditions \eqref{def:learning}, \eqref{def:sorting}
are expressed in terms of function $f(x,t)$ as
\begin{equation}
    \label{def:learning f}
    \liminf_{t\to\infty} \intO{p(x)f(x,t)} \geq \frac{M_c-1}{M},\quad \limsup_{t\to\infty} \intO{p(x)f(x,t)} \leq \frac{M_c}{M},
\end{equation}
and 
\begin{equation}
    \label{def:sorting f}
     \lim_{t\to\infty} \int_{-R}^R f(x,t)\,\dx =0,\quad \forall R>0,
     \end{equation}
respectively.

%%%%%%%%%%%%%%%%%%%%%%%%%%%%%%%%%%%%%%%%%%%%%%%%%%%%%%%%%%%%%%

\section{Existence results}
\label{sec:existence}
In this section, we show the existence and uniqueness of the strong solutions to the Cauchy problem for  
\begin{equation}
\label{eq:f2}
\left\{
\begin{array}{l}
\partial_t f + \frac{(M-1)}{\sqrt{\tau}}a(t)\partial_x(pf) - \frac{(M-1)^2}{2}\left(a^2(t)+\frac{1}{M-1}b(t)\right)\partial^2_{xx}(pf){}={}0,\quad (x,t)\in\mathbb{R}\times\mathbb{R}^+,\\
f(x,0)=f_0(x),
\end{array}\right.
\end{equation}
where
   \begin{equation}
        \label{def:a}
      a(t) = \intO{(\kappa - p(x))f(x,t)},
    \end{equation}
and
\begin{equation}
\label{def:b}
b(t) = \intO{p(x)f(x,t)}\intO{(1-p(x))f(x,t)}.
\end{equation}

We start with some a-priori estimates verified by a generic sufficiently smooth solution to motivate our approach and further regularizations of the system. Existence of solution satisfying these a-priori estimates is proven in Theorem \ref{thm:main1}.

\subsection{A-priori estimates}
\label{sec:apriori estimates}
%In this section we provide some a-priori estimates for solutions  of \eqref{eq:f2}. 

%The equation is supplemented with the initial value function $f(0,\cdot)=f_0(\cdot).$ 
We assume
that the weight function $p$ satisfies:
\eq{
\label{cond:p_apriori}
& p\in C^2(\R),\\
& p(x)\in (p_{min},1],\quad \forall x, \\
& 0\leq p'(x)\leq c_p(1-p(x)),\quad
|p''(x)|\leq c_p(1-p(x)),\quad \forall x,
}
for some $p_{min}\geq 0$ and some constant $c_p>0$.

We have the following result.
\begin{lemma}\label{lem:1_apriori}
   Let $M>1$, $\tau>0$, $\kappa\in(0,1)$. Let $p(x)$ satisfy the conditions \eqref{cond:p_apriori}, and let $f_0(x)\geq 0$ be such that 
    \eq{\intO{f_0(x)}=1,\quad \intO{p(x)f_0^2(x)}<\infty.}
    
    Then, every sufficiently regular $f$ 
    satisfying $f\geq 0$, $f(\cdot,t)=f_0(\cdot)$ and \eqref{eq:f2} with $a(t)$  and $b(t)$ given through \eqref{def:a}  and \eqref{def:b} satisfies
\eq{\label{eq:mass}
\intO{f(x,t)}=\intO{f_0(x)}=1,\quad \forall t\in[0,T],}
\eq{\label{ab_leq}
|a(t)|< 1, \quad 0\leq b(t)\leq 1,\quad \forall t\in[0,T].
}
Moreover,
\eq{\label{eq:energy}
&\sup_{t\in[0,T]} \intO{p(x)f^2(x,t)}+\int_0^T \frac{(M-1)^2}{2}\left(a^2(t)+\frac{1}{M-1}b(t)\right) \intO{|\partial_x (p(x)f(x,t))|^2}\,\dt\\
&\hspace{7cm}\leq \intO{p(x)f_0^2(x)}.
}    
\end{lemma}

\begin{proof}
    The first equality follows from integration of equation \eqref{eq:f2} over $(0,t)\times \R$ and use of the fact that $f, \px f\to 0$ as $|x|\to\infty$. Estimates \eqref{ab_leq} then follow directly from definitions \eqref{def:a} and \eqref{def:b}, conditions \eqref{cond:p_apriori} and assumption $\kappa\in(0,1)$.
\\
    The inequality \eqref{eq:energy} is obtained by multiplication of \ref{eq:f2} by $p(x)f(x,t)$ and integration by parts, using again the far-field boundary conditions.
\end{proof}

The next set of a-priori estimates, concerns quantities $\alpha(t)$ and $\beta(t)$:
\eq{
\label{def:alpha_beta}
\alpha(t)=\intO{p(x) f(x,t)},\quad \beta(t) = \intO{(1-p(x))p(x)f(x,t)}.
}
We have the following result.
\begin{lemma}\label{lem:2_apriori}
    Let assumptions of Lemma \ref{lem:1_apriori} be satisfied. 
    
    There is $c_0$ that depends only on $c_p,$ $\tau$ and $M,$ such that 
\begin{equation}
    \label{est:alpha'1}
    |\alpha'(t)|\leq c_0\beta(t),\quad |\beta'(t)|{}\leq{}c_0\beta(t),\quad t>0.
\end{equation}
In particular,
\begin{equation}
    \label{eq:lower_bound_beta1}
    \beta(t)\geq \beta(0)e^{-c_0t},\quad t>0.
\end{equation}      
\end{lemma}
\begin{proof}
    From \eqref{eq:mass} we obtain
    \eq{\label{sup:albe}
    \sup_{t\in[0,T]}\alpha(t)\leq 1\quad
\text{and} \quad \sup_{t\in[0,T]}\beta(t)\leq 1.}
Because for all $x\in\R$, $\kappa-p(x)\in (-1,1)$, it follows that 
\eq{\label{sup:ab}
\sup_{t\in[0,T]}|a(t)|< 1\quad 
\text{and} \quad \sup_{t\in[0,T]}b(t)\leq 1.}

Multiplying  \eqref{eq:f2} by $p(x)$ and integrating over $\R$ we obtain the equation on $\alpha'(t)$
\eq{
\alpha'(t)=- \frac{(M-1)}{\sqrt{\tau}}a(t)\intO{p\partial_x(pf)} + \frac{(M-1)^2}{2}\left(a^2(t)+\frac{1}{M-1}b(t)\right)\intO{p\partial^2_{xx}(pf)}.
}
Integrating by parts, we obtain
\eq{
\alpha'(t)=\frac{(M-1)}{\sqrt{\tau}}a(t)\intO{ p'\, p f} + \frac{(M-1)^2}{2}\left(a^2(t)+\frac{1}{M-1}b(t)\right)\intO{ p''\,  pf}.
}
Therefore, from \eqref{cond:p_apriori} we obtain
\eq{|\alpha'(t)|\leq C \lr{|a(t)|+a^2(t)+b(t)}\beta(t),}
and thus, the first part of \eqref{est:alpha'1} follows from \eqref{sup:albe} and \eqref{sup:ab}.

To write equation for $\beta(t)$, we multiply \eqref{eq:f2} by $p(x)(1-p(x))$ and integrate over $\R$. After integration by parts we get
\eq{
\beta'(t)=&\frac{(M-1)}{\sqrt{\tau}}a(t)\intO{[ p(1-p)]'\, p f} \\
&\quad+ \frac{(M-1)^2}{2}\left(a^2(t)
+\frac{1}{M-1}b(t)\right)\intO{ [p(1-p)]''\,  pf}.
}
As before, we can use \eqref{cond:p_apriori} to deduce the second part of estimate \eqref{est:alpha'1}. 
% \red{Add details?}
The estimate \eqref{eq:lower_bound_beta1} follows just by integration.
\end{proof}

\begin{remark}
    In particular, it follows from the above lemma that
\eq{\label{b_to_beta}
b(t) &= \intO{p(x) f(x,t)}-\left(\intO{p(x) f(x,t)}\right)^2\\
&\geq 
 \intO{p(x) f(x,t)}-\intO{\lr{p(x)}^2 f(x,t)}\\
 &=\intO{ p(x)(1-p(x))f(x,t)}=\beta(t) >0.}
 Thus, equation \eqref{eq:f2} is parabolic. 
\end{remark}

\subsection{Main existence result}
From the previous section, it follows that \eqref{eq:f2} is a parabolic problem with possibly degenerate coefficients at $x=-\infty$, depending on the value of $p_{min}$. It is therefore more convenient to switch to a new function $g=pf$ and consider the equation
\begin{equation}
\label{eq:g1}
\partial_t g + c(t)p(x)\partial_x g - d(t)p(x)\partial^2_{xx}g{}={}0,\quad (x,t)\in\mathbb{R}\times\mathbb{R}^+,
\end{equation}
where we denoted
\begin{equation}
\label{def:cd}
c(t){}={}\frac{(M-1)a(t)}{\sqrt{\tau}},\qquad d(t)=\frac{(M-1)^2}{2}\left(a^2(t)+\frac{1}{M-1}b(t)\right).
\end{equation}
Recalling the definition of $\alpha$ \eqref{def:alpha_beta},
one can now write
\eq{\label{alpha_g}
\alpha(t)=\intO{g(x,t)}
}
and
\begin{equation}
\label{def:ab}
a(t)=a(\alpha(t))=\kappa - \alpha(t),\quad b(t)=b(\alpha(t))=\alpha(t)(1-\alpha(t)).
\end{equation}
In view of above, equation \eqref{eq:g1} can be interpreted as a semi-linear (parabolic) equation for $g(t)$.

We also supplement \eqref{eq:g1} with the initial conditions 
\eq{\label{ini:g0}
g(x,0)=g_0(x)= p(x) f_0(x).}
The purpose of this section is to prove the following result.
\begin{theorem}\label{thm:main1}
   Let $M>1$, $\tau>0$, $\kappa\in(0,1)$. Let $p(x)$ satisfy the conditions \eqref{cond:p_apriori}, and let $g_0\in H^1(\mathbb{R})$ be such that $g_0\geq 0$ and
\eq{\label{g0}
\intO{\frac{g_0}{p}}=1, \quad 
\intO{\frac{e^{4|x|}g_0^2}{p}}<+\infty, \quad
\intO{e^{4|x|}|\partial_xg_0|^2}<+\infty.
}
Then, there exists a unique strong solution to problem \eqref{eq:g1}-- \eqref{ini:g0}, such that  
for any $T>0,$
\[
g\in L^\infty((0,T);H^1(\mathbb{R}))\cap C([0,T];L^2(\mathbb{R}))\cap L^2((0,T);H^2(\mathbb{R})),
\]
\[
\frac{e^{2|x|}g}{\sqrt{p}},\,
e^{2|x|}\partial_x g{}\in{} L^\infty((0,T); L^2(\mathbb{R})),
\]
\[
\frac{e^{2|x|}\partial_tg}{\sqrt{p}},\,
e^{2|x|}\sqrt{p}\partial_{xx}^2g{}\in{}L^2((0,T); L^2(\mathbb{R})).
\]
Moreover,
\begin{enumerate}
\item[(i)] for all $t\in[0,T]$,
\eq{g(\cdot,t)\geq 0,\qquad \intO{\frac{g(x,t)}{p(x)}}=1.
}
In particular, $\intO{g(x,t)}\leq 1,$ for all $t\in[0,T]$.

\item[(ii)] For any $T>0,$ there is  $C=C(T)>T$ such that
\begin{equation}
\label{est:non_lin_inf}
\mbox{\rm ess}\sup_{t\in[0,T]} \| \frac{e^{2|x|}g(\cdot,t)}{\sqrt{p}}\|_{L^2(\mathbb{R})}
{}+{}
\|e^{2|x|}\partial_x g\|_{L^2(Q_T)}{}
\leq{}
C\| \frac{e^{2|x|}g_0}{\sqrt{p}}\|_{L^2(\mathbb{R})},
\end{equation}
and moreover
\eq{\label{est:non_lin_exp}
\mbox{\rm ess}\sup_{t\in[0,T]} \| e^{2|x|}\partial_x g(\cdot,t)\|_{L^2(\mathbb{R})}^2
{}+{}
\intTO{\frac{e^{4|x|}|\partial_t g|^2}{p}}+    \intTO{p e^{4|x|}|\partial_{xx}^2g|^2}\\
\leq C \left(
        \intO{\frac{e^{4|x|}g^2_0(x)}{p}}{}+{}
        \intO{e^{4|x|}|\partial_xg_0(x)|^2}\right).
}

\end{enumerate}

\end{theorem}

\begin{remark}
    The assumption for the exponential decay of initial data at infinity is quite restrictive, and in fact could be replaced by polynomial decay. The exponential decay plays a role only in the further parts of our manuscript devoted to the long-time behavior.
\end{remark}

\begin{remark}
     As an immediate consequence of Theorem \eqref{thm:main1}, we recover the solution $f(x,t)$ to the equation in its original form \eqref{eq:f2}. Indeed, as $p(x)>0$ we can define
\[f(x,t):=\frac{g(x,t)}{p(x)}.\]
\end{remark}

The proof of Theorem \ref{thm:main1} follows in several steps. We first regularize the system by modifying $p(x)$ close to zero, and we also freeze the coefficients $c(t)$, $d(t)$ linearizing the equation. 
The existence result is obtained at this level, following the approach of Krylov \cite{Krylov}. Following this step, recovery of the solution to equation \eqref{eq:g1} is obtained by the fixed point argument and a subsequent limit passage with the regularization parameter $\ep$.

\subsection{Existence of solutions to the linearized and regularized system}

In this section we first regularize the system \eqref{eq:g1} by considering an approximation of function function $p(x)$. 
By $p_\e(x)$ we denote sequence of monotonically increasing $C^2(\mathbb{R})$ functions, given by 
 \eq{\label{p_ep_cond}
 p_\e=\frac{p(x)+\ep}{1+\ep}\geq \frac{p_{min}+\ep}{1+\ep}.
 }
 Moreover,  each $p_\e$ verifies
\begin{equation}
\label{cond:p1_ep}
|p_\e'(x)|\leq c_p(1-p_\e(x)),\quad
|p_\e''(x)|\leq c_p(1-p_\e(x)),
\end{equation}
for some $c_p>0$ independent of $\e.$ 
% \red{Later, I'd like to change $c_0$ and $c_p$ for capital letters to distinguish constants from function $c(t)$ that we already use.}

% Moreover, we assume $p_\e$ converges in $C^2(\mathbb{R})$ to a probability function $p$ (\red{write conditions for $p$}), and we assume  that $p_\e$ is a monotone sequence: $p_{\e_1}(x)\leq p_{\e_2}(x),$ when $\e_1\leq \e_2.$ \red{where do we use this?}

% \red{What about the initial conditions, what do we assume for $f_0$?}

We now linearize the system \eqref{eq:g1} by considering $\tilde\alpha(t)$ to be a given continuous function such that $b(\tilde\alpha(t))>0$. The coefficients $c(t)$ and $d(t)$ are then also fixed and given in terms  of $\tilde\alpha(t)$ via formulas \eqref{def:cd} and \eqref{def:ab}.

Our final regularised and linearised problem reads
\begin{equation}
\label{eq:g1_lin}
\left\{
\begin{array}{l}
\partial_t g + \tilde{c}(t)p_\e(x)\partial_x g - \tilde{d}(t)p_\e(x)\partial^2_{xx}g{}={}0,\quad (x,t)\in\mathbb{R}\times\mathbb{R}^+\\
g(x,0)=g_0.
\end{array}\right.
\end{equation}

For simplicity we omitted the dependence on $\tilde\alpha(t) $ denoting  $\tilde c(t)=c(\tilde \alpha(t))$, and similarly for $\tilde d$,  and dropped the index $\e$ denoting $g=g^\e$. In what follows, $Q_T{}={} \mathbb{R}\times[0,T].$

The existence and uniqueness of solutions to the Cauchy problem  \eqref{eq:g1_lin} follow from the classical results on the parabolic equations. Below we recall a relevant result for the non-homogenous linear parabolic equations with zero initial data,  see Theorem 5.1 from \cite{Krylov}. 
\begin{theorem}
\label{th:Krylov}
Let $T>0$ be arbitrary large, and let $p_\ep(x)$ satisfy \eqref{p_ep_cond}, \eqref{cond:p1_ep}. Let $F\in L^2(Q_T).$ There exists $\lambda_0>0$ and $C>0$ depending on $T,$ $\inf_{Q_T}\{\tilde d(t)p_\e(x)\},$ and $\sup_{t\in[0,T]}\{|\tilde c(t)|+\tilde d(t)\},$ such that if $\lambda>\lambda_0$ there  is a strong solution $\hat g$ of 
\eq{
\label{eq:g2}
&\partial_t \hat g +\tilde c(t)p_\e(x)\partial_x \hat g +\lambda \hat g - \tilde d(t)p_\e(x)\partial^2_{xx}\hat g=F,\quad (x,t)\in Q_T,\\
&\hat g(x,0)=0,
}
with 
\eq{\label{reg_hat}
\|\hat g\|_{L^2(Q_T)}{}+{}\|\partial_t \hat g\|_{L^2(Q_T)}{}+{}\|\partial_x \hat g\|_{L^2(Q_T)}{}+{}\|\partial_{xx}\hat g\|_{L^2(Q_T)}\leq C\|F\|_{L^2(Q_T)}.
}
The solution is unique in this regularity class.
\end{theorem}

We will rely on this theorem to prove the following result.
% \textcolor{blue}{E: In the below theorem, do we need in addition $g\in C([0,T];H^1(\mathbb R))$ like in Lemma 4.6 to use Grönwall for terms like $\Dt\!
%    \int_{\mathbb{R}}|\partial_x g|^{2}dx$?}
\begin{theorem}
\label{thm:linear}
Let $T>0$ be arbitrary large, and let $p_\ep(x)$ satisfy \eqref{p_ep_cond}, \eqref{cond:p1_ep}. Let
$g_0\in H^1(\mathbb{R})$ be such that $g_0\geq 0$ and
\eq{\label{g0_ep}
\intO{\frac{g_0}{p_\ep}}=1, \quad 
\intO{\frac{e^{4|x|}
g_0^2}{p_\ep}}<+\infty, \quad
\intO{e^{4|x|}|\partial_xg_0|^2}<+\infty.
}

Let $\tilde c(t),\ \tilde d(t)$ satisfy the conditions
\eq{
\inf_{t\in[0,T]}\{\tilde d(t)\}>0, \quad   \text{and}\quad  \sup_{t\in[0,T]}\{|\tilde c(t)|+\tilde d(t)+|d'(t)|\}\leq C<+\infty. 
}
Then, there exists a unique, strong solution $g(x,t)$ of the Cauchy problem for \eqref{eq:g1_lin} on {$Q_T$} with the following properties:
\begin{enumerate}
\item[(i)] For any $T>0,$
\[
g\in L^\infty((0,T);H^1(\mathbb{R}))\cap C([0,T];L^2(\mathbb{R}))\cap L^2((0,T);H^2(\mathbb{R})).
\]
\item[(ii)] For any $T>0,$ there is  $C=C(T)>0$,   such that  
\begin{equation}
\label{est:linear1}
\|g\|_{L^2((0,T);H^1(\mathbb{R}))}{}+{}\|\partial_t g\|_{L^2(Q_T)}
{}+{}\|\partial_{xx} g\|_{L^2(Q_T)}\leq C\|g_0\|_{H^1},
\end{equation}
and
\begin{equation}
\label{est:linear2}
\mbox{\rm ess}\sup_{t\in[0,T]}\|g(\cdot,t)\|_{H^1(\mathbb{R})}{}\leq{} C\|g_0\|_{H^1}.
\end{equation}
\item[(iii)] For all $t\in[0,T]$,
\eq{g(\cdot,t)\geq 0,\qquad \intO{\frac{g(x,t)}{p_\ep(x)}}=1.
}
In particular, $\intO{g(x,t)}\leq 1,$ for all $t\in[0,T]$.

\item[(iv)] For any $T>0,$ there is  $C=C(T)>T$ such that
\begin{equation}
\label{est:inf}
\mbox{\rm ess}\sup_{t\in[0,T]} \| e^{2|x|}g(\cdot,t)\|_{L^2(\mathbb{R})}
{}+{}\|e^{2|x|}\partial_x g\|_{L^2(Q_T)}{}
\leq{}
C\| e^{2|x|}g_0\|_{L^2(\mathbb{R})},
\end{equation}
and moreover
\eq{\label{iv_second}
\mbox{\rm ess}\sup_{t\in[0,T]} \intO{ e^{4|x|}|\partial_xg(\cdot,t)|^2}
{}+{}
\intTO{\frac{e^{4|x|}|\partial_t g|^2}{p_\e}}+    \intTO{p_\e e^{4|x|}|\partial_{xx}^2g|^2}\\
\leq C \left(
        \intO{\frac{e^{4|x|}g^2_0(x)}{p_\e}}{}+{}
        \intO{e^{4|x|}|\partial_xg_0(x)|^2}\right).
}
% \red{add the extra estimates \eqref{est:gt2} and \eqref{est:g63}}

\end{enumerate}
\end{theorem}

\begin{proof}
 From the conditions imposed on functions $\tilde d$ and $p_\e$ we have
\[
\inf_{Q_T}\{ \tilde d(t)p_\e(x)\}>0,\quad \forall T>0.
\]
Let $w$ be the solution of the heat equation
\begin{equation}
\label{eq:heat}
\partial_t w - \partial^2_{xx}w{}={}0,\quad w(\cdot,0)=g_0.
\end{equation}
The Cauchy problem \eqref{eq:heat} has a unique strong solution $w$ that verifies inequality
\eq{\label{heat_prop}
&\sup_{t\in(0,T)}\left( \|w(t)\|_{L^2(\mathbb{R})}+\|\partial_x w(t)\|_{L^2(\mathbb{R})}\right){}\\
&
\qquad\qquad\qquad+\|\partial_t w\|_{L^2(Q_T)} {}+{}\|\partial_x w\|_{L^2(Q_T)}{}+{}
\|\partial_{xx} w\|_{L^2(Q_T)}\leq \|g_0\|_{H^1}.
}
Substituting to \eqref{eq:g1_lin} $g(x,t)$ of the form
\begin{equation*}
g(x,t)= e^{\lambda t}(\hat{g}(x,t)+w(x,t)),
% \hat{g}=g(x,t)e^{-\lambda t} - w(x,t),
\end{equation*}
we obtain an equation \eqref{eq:g2} for $\hat{g}(x,t)$ with $F$ being a linear combination of function $w$, $\partial_t w, \partial_x w, \partial^2_{xx}w$ with coefficients that are bounded functions. Therefore, from \eqref{heat_prop}, we have 
\[
\|F\|_{L^2(Q_T)}{}\leq C\|g_0\|_{H^1},
\]
form some $C>0.$ Therefore, all  the assumptions of Theorem \ref{th:Krylov} are satisfied. 

We now take $\lambda>\lambda_0,$ with $\lambda_0$ as in Theorem \ref{th:Krylov}. For such $\lambda$'s Theorem \ref{th:Krylov} guaranties the existence of a unique solution $\hat{g}$ of
\eqref{eq:g2}. This implies the existence and uniqueness of solution $g$ of \eqref{eq:g1_lin} with regularity properties \eqref{reg_hat}. In particular, $g(x,t)$ and $\px g(x,t)$ tend to $0$ for $|x|\to\infty$. We now justify the estimates from the rest of the statement of Theorem \ref{thm:linear}.

Multiplying equation \eqref{eq:g1_lin} by ${g}/{p_\e}$ and integrating it over $Q_T,$ we obtain 
\begin{equation}
\label{est:g^2}
\sup_{t\in[0,T]} \intO{\frac{g^2}{p_\e}}+{}\int_0^T  \tilde d(t)\intO{|\partial_x g|^2}\dt{}\leq{}\intO{\frac{g_0^2}{p_\e}}.
\end{equation}
Next, multiplying  \eqref{eq:g1_lin} by $\partial_t g_\e/p_\e,$ integrate in $x$ we get
    \begin{multline}
        \intO{ \frac{|\partial_t g|^2}{p_\e}}+{}
        \Dt \intO{ \frac{\tilde d(t)}{2}|\partial_x g|^2}{}={}
        -\tilde c(t)\intO{ \partial_tg\partial_xg}
        -\frac{\tilde d'(t)}{2}\intO{ |\partial_xg|^2}\\
        {}\leq{}\frac{1}{2}\intO{\frac{|\partial_t g|^2}{p_\e}}{}+{}
        \frac{\sup_{(0,T)} (\tilde c^2(t)+|\tilde d'(t)|)}{2\inf_{(0,T)} \tilde d(t)}\intO{ \tilde 
d(t)|\partial_x g|^2}.
    \end{multline}
Integrating w.r.t. time, and using assumptions on $\tilde c(t), \tilde d(t)$ we therefore get
\begin{equation}
    \label{est:g_t}
        \intTO{ \frac{|\partial_tg|^2}{p_\e}}{}+{}\mbox{\rm ess}\sup_{t\in[0,T]}\left(\tilde d(t)\intO{|\partial_x g(x,t)|^2}\right){}\leq{} C \|\partial_xg_0\|^2_{L^2(\mathbb{R})}.
    \end{equation}
This implies  estimate \eqref{est:linear2} and part of \eqref {est:linear1}. To conclude, we observe that from \eqref{est:g^2}, \eqref{est:g_t} and equation \eqref{eq:g1_lin} it follows that 
\begin{equation}
\label{est:g_xx}
    \int_0^T \tilde d^2(t)\intO{p_\e|\partial_{xx}^2g|^2}\,\dt{}\leq{}C(T)\left(\|\partial_xg_0\|^2_{L^2(\mathbb{R})}{}+{}\|\frac{g_0}{\sqrt{p_\e}}\|^2_{L^2(\mathbb{R})}\right),
\end{equation}    
for any $T>0$ with some $C=C(T)>0$.

The estimates of part ({\emph{ii}}) imply the inclusion of $g$ in the functional spaces of part ({\emph i}).

Part ({\emph{iii}})  of Theorem \ref{thm:linear} follows from the  maximum principle and nonnegativity of the initial data. The conservation of $\intO{\frac{g(x,t)}{p_\ep(x)}}$ follows from the form of system \eqref{eq:g1_lin}, regularity of the solution and assumption on the data \eqref{g0_ep}.

To prove part ({\emph{iv}})  let us consider a cut-off function $\phi_L(x)$, which is even and defined for $x>0$ by
    \[
    \phi_L(x){}={}
    \left\{
    \begin{array}{ll}
    e^{4x} & 0\leq x<L,\\
    (L-x)+e^{4L} & L\leq x< L+e^{4L},\\
    0 & x\geq L+e^{4L}.
    \end{array}    
    \right.
    \]
\end{proof}
We now multiply \eqref{eq:g1_lin} by $\phi_L g/p_\e$ and integrate in $x$ to get
\eq{\label{eq:exp}
&\Dt\intO{\frac{\phi_L g^2}{p_\e}}{}+{}    \tilde d(t)\intO{\phi_L|\partial_xg|^2}{}\\
&=  -\tilde d(t)\intO{\px \phi_L g \partial_xg }
-\tilde c(t)\intO{\phi_L g \partial_xg }\\
&=  -\tilde d(t)\intO{\px \phi_L g \partial_xg }
+\frac{\tilde c(t)}{2}\intO{g^2\px \phi_L  }.
% &{}\leq{}\frac{1}{2}\tilde d(t)\intO{(1+x^2)|\partial_xg_\e|^2}
% {}+{}\Red{\lr{\frac{\tilde c^2(t)}{\tilde d(t)}+4\tilde d(t)}}\intO{\frac{(1+x^2)g^2_\e(x,t)}{p_\e}}\\
% &{}\leq{}\frac{1}{2}\tilde d(t)\intO{(1+x^2)|\partial_xg_\e|^2}
% {}+{}\frac{\sup_{t\in[0,T]}(\tilde  c^2(t)+2\tilde  d^2(t))}{\inf_{t\in[0,T]}\tilde d(t)}\intO{\frac{(1+x^2)g^2_\e(x,t)}{p_\e}}.
}
We have
\eqh{
\left|\intO{g^2\partial_x\phi_L}\right|{}&={}
\left|4\int_0^{L}g^2e^{4x}\,\dx
-4\int_{-L}^{0}g^2e^{-4x}\,\dx
{}-{}\int_L^{L+e^{4L}}g^2\,\dx
{}+{}\int_{-L-e^{4L}}^{-L}g^2\,\dx
\right|\\
{}&\leq{}
4\intO{\phi_Lg^2}
{}+{}\intO{g^2}
{}\leq{}
4\intO{\frac{\phi_Lg^2}{p_\ep}}
{}+{}\intO{\frac{g^2}{p_\ep}}\leq 5\intO{\frac{\phi_Lg^2}{p_\ep}}.
}
In a similar way, using the definition of $\phi_L$ and Young's inequality, we can get
\eq{
\label{eq:energy2.3}
&\left|\intO{g\partial_{x}g\partial_x\phi_L}
    \right|{}\leq{} \frac{1}{8} \intO{|\partial_{x}g|^2 |\partial_x\phi_L|}+2 \intO{g^2 |\partial_x\phi_L|}\\
&\leq\frac12\int_0^{L}|\partial_x g|^2e^{4x}\,\dx
+\frac12\int_{-L}^{0}|\partial_x g|^2e^{-4x}\,\dx
{}+\frac{1}{8}{}\int_L^{L+e^{4L}}|\partial_x g|^2\,\dx
{}+\frac{1}{8}{}\int_{-L-e^{4L}}^{-L}|\partial_x g|^2\,\dx
\\
&\qquad+8\int_0^{L}g^2e^{4x}\,\dx
+8\int_{-L}^{0}g^2e^{-4x}\,\dx
{}+{2}{}\int_L^{L+e^{4L}}g^2\,\dx
{}+{2}{}\int_{-L-e^{4L}}^{-L}g^2\,\dx\\
&\leq \frac{1}{2}\intO{|\partial_x g|^2\phi_L}+8 \intO{g^2\phi_L}+\frac18 \intO{|\partial_x g|^2} +2 \intO{g^2}\\
&\leq \frac{5}8{|\partial_x g|^2\phi_L}+10 \intO{\frac{\phi_Lg^2}{p_\ep}}.
}
Therefore, returning to \eqref{eq:exp}, we obtain
\eqh{
&\Dt\intO{\frac{\phi_L g^2}{p_\e}}{}+{}    \tilde d(t)\intO{\phi_L|\partial_xg|^2}{}\\
% &=  \textcolor{red}{-\tilde d(t)\intO{\px \phi_L g \partial_xg }}
% -\tilde c(t)\intO{\phi_L g \partial_xg }\\
% &=  \textcolor{red}{-\tilde d(t)\intO{\px \phi_L g \partial_xg }}
% +\frac{\tilde c(t)}{2}\intO{g^2\px \phi_L  }\\
&{}\leq{}\frac{5}{8}\tilde d(t)\intO{\phi_L|\partial_xg|^2}
{}+{} C\lr{|\tilde c(t)|+\tilde d(t)}\intO{\frac{\phi_L g^2}{p_\e}}.
% &{}\leq{}\frac{1}{2}\tilde d(t)\intO{(1+x^2)|\partial_xg_\e|^2}
% {}+{}\frac{\sup_{t\in[0,T]}(\tilde  c^2(t)+2\tilde  d^2(t))}{\inf_{t\in[0,T]}\tilde d(t)}\intO{\frac{(1+x^2)g^2_\e(x,t)}{p_\e}}.
}
Using the Gronwall argument and letting $L\to \infty$, we obtain
\eq{\label{4xg}
&\mbox{\rm ess}\sup_{t\in[0,T]}\intO{\frac{e^{4|x|}
g^2}{p_\ep}}{}+{}    \int_0^T\!\!\!\tilde d(t)\intO{e^{4|x|}|\partial_xg|^2}{}\,\dt
\leq C\intO{\frac{e^{4|x|}
g_0^2}{p_\ep}}.
}

% Note that we implicitly assumed that the boundary terms from the integration by parts are equal to zero. This can be justified using a suitable cutoff function. We omit the details here. \red{maybe move the part from long time asymptotics with exponential here?}
% The required estimate follows  by use of the Gronwall inequality. In particular, we have
%    \begin{equation}
%     \label{est:g^2(1+x^2)}
%     \mbox{\rm ess}\sup_{t\in[0,T]}\intO{\frac{(1+x^2)g^2_\e(x,t)}{p_\e}}{}+{}    \int_0^T\tilde d(t)\intO{(1+x^2)|\partial_xg_\e|^2}\,\dt{}\leq{}C\intO{\frac{(1+x^2)g^2_0(x)}{p_\e}}.
%     \end{equation}
In addition, multiplying equation \eqref{eq:g1_lin} by $\phi_L\partial_t g/p_\e,$ one can obtain 
\eqh{  &\intO{\frac{\phi_L|\partial_t g|^2}{p_\e}}{}+\Dt\lr{ \frac{\tilde d(t)}{2}\intO{\phi_L|\partial_x g(x,t)|^2}} \\
        &= \frac{\tilde d'(t)}{2}\intO{\phi_L|\px g|^2}-\tilde d(t)\intO{\px\phi_L \px g\pt g}-\tilde c(t)\intO{\phi_L\px g \pt g}\\
        &\leq C\frac{\sup_{t\in[0,T]} \lr{\tilde c^2(t)+\tilde d^2(t)+|\tilde d'(t)|}}{\inf_{t\in[0,T]}\tilde d(t)}\tilde d(t)\intO{\phi_L|\partial_x g(x,t)|^2}+ \frac12\intO{\frac{\phi_L|\partial_t g|^2}{p_\e}},
}
where we estimated $\px\phi_L$ like before in \eqref{eq:energy2.3}. Thus, by the Gronwall inequality, and letting $L\to\infty$ we obtain 
\eq{\label{est:gt2T}
\intTO{\frac{e^{4|x|}|\partial_t g|^2}{p_\e}}+{}\mbox{\rm ess}\sup_{t\in[0,T]}\left(\tilde d(t)\intO{e^{4|x|}|\partial_x g(x,t)|^2}\right){}
\leq C \intO{e^{4|x|}|\partial_x g_0(x)|^2}.
}
Finally, multiplying \eqref{eq:g1_lin} by $\phi_L\partial^2_{xx}g$, we verify
\eqh{
&\tilde d(t)    \intO{p_\e\phi_L|\partial_{xx}^2g|^2}{}\\
&\quad=\intO{\pt g \phi_L\partial^2_{xx}g}-\tilde c(t)\intO{p_\ep \px g\phi_L\partial^2_{xx}g}
\\
&\quad\leq \frac{\tilde d(t)}{2}    \intO{p_\e\phi_L|\partial_{xx}^2g|^2}
+\frac{C}{\inf_{t\in[0,T]}\tilde d(t)}\intO{\frac{\phi_L|\partial_t g|^2}{p_\e}}\\
&\qquad\qquad\qquad+\frac{C \sup_{t\in[0,T]} \tilde c^2(t)}{\inf_{t\in[0,T]}\tilde d(t)}\tilde d(t)\intO{\phi_L|\partial_x g(x,t)|^2}.
}
Thus, integrating in time, and letting $L\to\infty$, we obtain
\eq{\label{4xgx}
\int_0^T\!\!\!\tilde d(t)    \intO{p_\e e^{4|x|}|\partial_{xx}^2g|^2}\dt\leq{}C\left(
        \intO{\frac{e^{4|x|}g^2_0(x)}{p_\e}}{}+{}
        \intO{e^{4|x|}|\partial_xg_0(x)|^2}\right),}
which concludes the proof of \eqref{iv_second} and of the whole theorem. $\Box$

\subsection{Fixed point argument}

In this section, we construct the solution from theorem \ref{thm:main1}.

To obtain a solution to \eqref{eq:g1}-\eqref{ini:g0}, we  apply a fixed point argument to a map ${\cal T}[\tilde\alpha]$ defined by 
\eq{\label{Tmap}
{\cal T}[\tilde\alpha](t){}=\alpha(t),
}
where 
\eq{\label{alpha_again}
\alpha(t)=\intO{g(x,t)},} and 
$g$ satisfies \eqref{eq:g1_lin}. 

We define
\eq{
\beta(t) = \intO{(1-p_\e(x))g(x,t)},
}
and prove the analogue of Lemma \ref{lem:2_apriori}.

\begin{lemma}
\label{lemma:moments_bounds}
Let assumptions of Theorem \ref{thm:linear} be satisfied and let $g(x,t)$ be the corresponding solution  of the Cauchy problem \eqref{eq:g1_lin}.

There exists $c_0$ depending only on $c_p$ from \eqref{cond:p1_ep}, such that 
\begin{equation}
    \label{est:alpha'}
    |\alpha'(t)|\leq c_0,\quad |\beta'(t)|{}\leq{}c_0\beta(t),\quad t>0.
\end{equation}
In particular,
\begin{equation}
    \label{eq:lower_bound_beta}
    \beta(t)\geq \beta(0)e^{-c_0t}>0,\quad t>0.
\end{equation}    
  
\end{lemma}

\begin{proof}
Let $g$ be the solution of (\ref{eq:g1_lin}) obtained in Theorem \ref{thm:linear} corresponding to $g_0.$ 
From part ({\emph{iii}})  of Theorem \ref{thm:linear}   it follows that 
$\sup_{t>0}\alpha(t)\leq 1$
and $\sup_{t>0}\beta(t)\leq 1.$

The rest of the proof follows similarly to proof of Lemma \ref{lem:2_apriori}. Integrating equation \eqref{eq:g1}  over $x$ and using 
\eqref{cond:p1_ep}, we obtain estimate on $|\alpha'(t)|.$ 
Multiplying  (\ref{eq:g1_lin}) by $(1-p_{\ep})$ and integrating, we get:
\[\Dt\intO{(1-p_\ep)g}+ \tilde c(t)\intO{p_\ep(1-p_\ep)\partial_x g} - \tilde d(t)\intO{p_\ep (1-p_\ep) \partial^2_{xx}g}={}0.\]
Integrating by parts,
\[\Dt\intO{(1-p_\ep)g}
=\tilde c(t)\intO{\left[p_\ep(1-p_\ep)\right]' g}+ \tilde d(t)\intO{\left[p_\ep (1-p_\ep)\right]'' g}.\]
We have that 
\eq{\left|\left[p_\ep(1-p_\ep )\right]'\right|=\left|p_\ep'(1-p_\ep)-p_\ep p_\ep'\right|
\leq |p_\ep'||(1-p_\ep)||+|p_\ep|| p_\ep'|
\leq C|p_\ep'|,}
where the last inequality uses $(1-p_\ep)\leq C$ and $p_\ep \leq C.$ Using (\ref{cond:p1_ep}), we get
\[|\left[p_\ep(1-p_\ep )\right]'|\leq c_0(1-p_\ep),\]
for some $c_0$ as in the statement of the lemma.
Similarly, 
\eqh{\left|\left[p_\ep(1-p_\ep )\right]''\right|=\left|p_\ep''(1-p_\ep)+2(p_\ep')^2-p_\ep p_\ep''\right|
\leq C(1-p_\ep ),}
where the last inequality uses $(1-p_\ep)\leq C,$  $p_\ep \leq C,$ and (\ref{cond:p1_ep}).

Since $\tilde c(t)$ and $\tilde d(t)$ are bounded,  the above bounds imply \eqref{est:alpha'}. In particular
\[ \beta'(t) \geq -c_0\beta(t),\]
which finishes the proof of \eqref{eq:lower_bound_beta}.
\end{proof}
\begin{corollary}\label{col:d}
Let $d(t)$ be defined through \eqref{def:cd}, \eqref{def:ab} where $\alpha(t)=\intO{g(x,t)}$ and $g(x,t)$ is the unique solution of \eqref{eq:g1_lin}. Then 
\begin{equation*}
    d(t)\geq \frac{(M-1)\beta(0)}{2}e^{-c_0t}>0,\quad t>0.
\end{equation*}  
\end{corollary}
\begin{proof}
    This estimate follows from estimate \eqref{eq:lower_bound_beta} and argument similar to \eqref{b_to_beta}. Note, that $\beta(0)=\intO{(1-p(x))g_0(x)}>0.$
\end{proof}

% \begin{equation}
% \label{eq:g1_ep}
% \partial_t g + c(\tilde\alpha(t))p_\e(x)\partial_x g - d(\tilde\alpha(t))p_\e(x)\partial^2_{xx}g{}={}0,\quad (x,t)\in\mathbb{R}\times\mathbb{R}^+,
% \end{equation}
% for which  functions $c(\tilde\alpha(t))$ and $d(\tilde\alpha(t))$ depend on  $\tilde\alpha(t)$ through \eqref{def:cd} and \eqref{def:ab}. 

We are now ready to state our main result to the nonlinear problem.

\begin{theorem}
\label{th:fixed_point}
Let $T>0$ be arbitrary large, and let $p_\ep(x)$ satisfy \eqref{p_ep_cond}, \eqref{cond:p1_ep}. Let
$g_0\in H^1(\mathbb{R})$ be such that $g_0\geq 0$ and
\eqh{
%\label{g0_ep}
\intO{\frac{g_0}{p_\ep}}=1, \quad 
\intO{\frac{e^{4|x|}
g_0^2}{p_\ep}}<+\infty, \quad
\intO{e^{4|x|}|\partial_xg_0|^2}<+\infty.
}
There exists a strong solution $g(x,t)$ of 
\eq{
\label{eq:g1_ep}
&\partial_t g + c(t)p_\e(x)\partial_x g - d(t)p_\e(x)\partial^2_{xx}g{}={}0,\quad (x,t)\in Q_T\\
&g(x,0)=g_0(x),
}
with  $c(t),\,d(t)$ given by \eqref{def:cd}, \eqref{def:ab}, and \eqref{alpha_g}. 

Moreover, the solution verifies the estimates \eqref{est:g^2_2}--\eqref{est:g_xx_2} below and is unique in the regularity class defined by these estimates. 
\end{theorem}
\begin{proof} 
We aim at using  the Schauder Fixed Point theorem for the map  ${\cal T}$ defined in \eqref{Tmap}. To this purpose let us consider the following closed convex subset $K$ of $C([0,T])$
\eq{
K=\{\alpha(t): \forall t\in[0,T],\ \alpha(t)\leq 1,\  {d}(\alpha(t))\geq \beta(0)e^{-c_0t}, \, |\alpha'(t)| \leq c_0, \}
}
where, $b(\alpha(t))$ is defined in \eqref{def:ab}, and $\beta(0)$ and $c_0$ are as in Lemma \ref{lemma:moments_bounds}.

% \red{I removed the assumption about weak differentiability as we assume that the derivative of $\alpha$ is bounded}

Let $\tilde \alpha\in K$, and $\tilde c= c(\tilde\alpha), \tilde d=\tilde d (\tilde\alpha)$ defined by $\eqref{def:cd}$ and $\tilde a=a(\tilde\alpha),\tilde b=b(\tilde\alpha)$ defined by \eqref{def:ab}.
From Corollary \ref{col:d} it follows that $d(\tilde\alpha(t))> C$ for any $T$ finite. Moreover, $\sup_{(0,T)}| c(\tilde\alpha(t))|,$ $\sup_{(0,T)}| d(\tilde\alpha(t))|,$ and $\sup_{(0,T)}| d'(\tilde\alpha(t))|,$ are bounded.
Therefore, from Theorem \ref{thm:linear}, Lemma \ref{lemma:moments_bounds} and Corollary \ref{col:d}, one finds that for $\alpha$  defined in \ref{alpha_again},   $\alpha={\cal T}[\tilde \alpha]\in K.$

Let us now show the continuity of ${\cal T}$.

Let $\tilde \alpha_1, \tilde\alpha_2\in K$ and $g_1,g_2$ be the corresponding solutions of the PDE \eqref{eq:g1_lin}.
Denote ${\alpha}_i=T[\alpha_i],\,i=1,2.$ We set $w=g_1-g_2$ and write the PDE for $w:$
\begin{equation}
\label{eq:w}
\partial_t w+\tilde c_1p_\e\partial_x w
-\tilde d_1p_\e\partial_{xx}^2
w{}={}(\tilde c_2-\tilde c_1)p_\e\partial_xg_2 {}-{}(\tilde d_2-\tilde d_1)p_\e\partial_{xx}^2g_2.
\end{equation}
Multiplying this equation by $w/p_\e$ and integrating in $x,$ we get:
\begin{multline}
\label{est:w1}
   \Dt\intO{\frac{|w|^2}{2p_\e}}
    {}+{}\tilde d_1\intO{ |\partial_x w|^2}={}
    (\tilde c_2-\tilde c_1)\intO{w\partial_xg_2}-\intO{(\tilde d_2-\tilde d_1)\int w\partial^2_{xx}g_2}\\
    {}\leq{} \intO{\frac{|w|^2}{2p_\e}} {}+{}|\tilde c_2-\tilde c_1|^2\intO{ p_\e|\partial_xg_2|^2}{}+{}|\tilde d_2-\tilde d_1|^2\intO{ p_\e|\partial^2_{xx}g_2|^2}\\
    {}\leq{}\intO{\frac{|w|^2}{2p_\e}} {}+{}C\sup_{[0,T]}|\tilde\alpha_2(t)-\tilde\alpha_1(t)|^2\left(\intO{p_\e|\partial_xg_2|^2}{}+{}\intO{p_\e|\partial^2_{xx}g_2|^2}\right).
\end{multline}
Using the estimates \eqref{est:linear1} for $g_2$ we conclude from here that there is $C=C(T),$ such that
\[
\sup_{t\in[0,T]}\|w(\cdot,t)\|_{L^2(\mathbb{R})}\leq C\|\tilde \alpha_2-\tilde \alpha_1\|_{C^0([0,T])}.
\]

For $t\in[0,T]$ we have for any $R>0,$
\begin{multline}
\label{est:alpha_diff}
|{\alpha}_1(t)-{\alpha}_2(t)|\leq\int |w(x,t)|\,dx{}\leq \int_{|x|>R}|w|\,dx
{}+{}\int_{|x|<R}|w(x,t)|\,dx\\
\leq
\frac{1}{e^{2R}}\sup_{[0,T]}\|e^{2|x|}w(\cdot,t)\|_{L^2(\mathbb{R})}
{}+{}\sqrt{R}\sup_{[0,T]}\|w(\cdot,t)\|_{L^2(\mathbb{R})}\\
{}\leq{}
\frac{C(T)}{e^{2R}}{}+{}C(R,T)\|\tilde \alpha_1-\tilde \alpha_2\|_{C^0([0,T])},
\end{multline}
where we used estimate \eqref{est:inf} for both $g_1$ and $g_2$.
The estimate \eqref{est:alpha_diff} implies the continuity of ${\cal T}$. Therefore, the existence of the fixed point might be now deduced using the Schauder Fixed Point Theorem, and so the existence of the solution to \eqref{eq:g1_ep} is proven.

Knowing now that there exists $\alpha\in K$ such that $\alpha={\cal T}[\alpha]$, we can essentially repeat the estimates from points $(ii)$, $(iii)$ and $(iv)$ of Theorem \ref{thm:linear}. The estimates 
\eqref{est:g^2}, \eqref{est:g_t}, \eqref{est:g_xx}, \eqref{4xg}, \eqref{est:gt2T} and \eqref{4xgx} together with Lemma \ref{lemma:moments_bounds} and Corollary \ref{col:d} imply that that the 
$L^2$ norms of the solution $g$ and its derivatives verify the following set estimates with $C$ independent of $\alpha$: 
\begin{equation}
\label{est:g^2_2}
\mbox{\rm ess}\sup_{t\in[0,T]} \intO{\frac{e^{4|x|}g^2}{p_\e}}{}+C\beta(0)e^{-c_0T}\intTO{e^{4|x|}|\partial_x g|^2}{}\leq{}C\left\|\frac{e^{2|x|}g_0}{\sqrt{p_\ep}}\right\|^2_{L^2(\mathbb{R})},
\end{equation}
\eq{
    \label{est:g_t_2}
        \intTO{\frac{e^{4|x|}|\partial_tg|^2}{p_\e}}{}+{}C\beta(0)e^{-c_0T}\mbox{\rm ess}\sup_{t\in[0,T]}\intO{e^{4|x|}|\partial_x g(x,t)|^2}{}\\\leq{} C \|e^{2|x|}\partial_xg_0\|^2_{L^2(\mathbb{R})},
}
\begin{equation}
\label{est:g_xx_2}
    C\beta(0)e^{-c_0T}\intTO{ p_\e e^{4|x|}|\partial_{xx}^2g|^2}{}\leq{}C\left(\|e^{2|x|}\partial_xg_0\|^2_{L^2(\mathbb{R})}{}+{}\left\|\frac{e^{2|x|}g_0}{\sqrt{p_\ep}}\right\|^2_{L^2(\mathbb{R})}\right),
\end{equation}
% and
% \begin{equation}
%     \label{est:g^2(1+x^2)_2}
%     \mbox{\rm ess}\sup_{t\in[0,T]}\intO{\frac{e^{4|x|}g^2(x,t)}{p_\e}}{}+C\beta(0)e^{-c_0T}    \intTO{e^{4|x|}|\partial_xg|^2}{}\leq{}C\intO{\frac{e^{4|x|}g^2_0(x)}{p_\e}},
%     \end{equation}

% In addition, multiplying equation \eqref{eq:g1_ep} by $(1+x^2)\partial_t g/p_\e,$ one can obtain the estimate
% \begin{multline}
%     \label{est:g_t_2(1+x^2)}
%         \intTO{\frac{e^{4|x|}|\partial_tg|^2}{p_\e}}{}+{}\mbox{\rm ess}\sup_{t\in[0,T]}\intO{e^{4|x|}|\partial_x g(x,t)|^2}
%         {}\leq{} C
%         \intO{e^{4|x|}|\partial_xg_0|^2(x)},
%     \end{multline}
% and multiplying \eqref{eq:g1_ep} by $\partial_{xx}^2g$, we get
% \begin{equation}
% \label{est:g_xx_2(1+x^2)}
%     \intTO{p_\e(1+x^2)|\partial_{xx}^2g|^2}{}\leq{}C\left(
%         \intO{\frac{(1+x^2)g^2_0(x)}{p_\e}}{}+{}
%         \intO{(1+x^2)|\partial_xg_0|^2(x)}\right).
% \end{equation}

% For $\alpha\in K,$ define a map $T[\alpha]$ by 
% \[
% T[\alpha](t){}={}\int g(x,t)\,dx.
% \]
%  Let us show that $T$ is continuous. Then, $T$ has a fixed point in $K.$

Now we show that the solution in this regularity class is unique.

Let $g_1,$ $g_2$ be two such solutions corresponding to the same initial data $g_0.$ Denote by $\alpha_i,c_i,d_i,$ $i=1,2,$ the corresponding coefficients in \eqref{eq:g1_ep}.
The difference $w=g_1-g_2$ verifies equation \eqref{eq:w} and estimate \eqref{est:w1}.
From that estimate, we find
\[
\sup_{t\in[0,T]}\|w(\cdot,t)\|^2_{L^2(\mathbb{R})}\leq C \|\alpha_1-\alpha_2\|^2_{C^0([0,T])}
\intTO{p_\ep|\partial_x g_2|^2+p_\ep|\partial_{xx}^2g_2|^2},
\]
for some non-decreasing $C=C(T).$

Multiplying equation \eqref{eq:w} by $e^{4|x|}w/p_\e,$ in a similar fashion, we get
\[
\sup_{t\in[0,T]}\|e^{2|x|}w(\cdot,t)\|^2_{L^2(\mathbb{R})}\leq C \|\alpha_1-\alpha_2\|^2_{C^0([0,T])}
\intTO{p_\e e^{4|x|}|\partial_x g_2|^2+p_\e e^{4|x|}|\partial_{xx}^2g_2|^2}.
\]
Using the last two estimates in \eqref{est:alpha_diff} and setting $R=1,$ we obtain
\[
\|\alpha_1-\alpha_2\|^2_{C^0([0,T])}
{}\leq{}C\|\alpha_1-\alpha_2\|^2_{C^0([0,T])}
\intTO{p_\e e^{4|x|}|\partial_x g_2|^2+p_\e e^{4|x|}|\partial_{xx}^2g_2|^2}.
\]
By the uniform continuity of the integral 
\[
\int_0^t\!\!\!\intO{p_\e e^{4|x|}|\partial_x g_2|^2+p_\e  e^{4|x|}|\partial_{xx}^2g_2|^2}\,{\rm d}s,\quad t\in[0,T],
\]
the last estimate implies that $\alpha_1(t)=\alpha_2(t),$ one some interval $t\in [0,\Delta],$ and then, by extension, on all subsequent intervals of length $\Delta,$ up to time $T.$ 

\end{proof}

\subsection{Passage to the limit $\e\to 0$.}
We now investigate the limit $\e\to0$ for the solution of the equation \eqref{eq:g1_ep}, which will allow us to conclude the proof of Theorem \ref{thm:main1}. 

Let $g_\e$ denote the unique solution obtained in Theorem \ref{th:fixed_point}. Then, uniformly w.r.t. $\e$, we have
\eq{g_\ep(\cdot,t)\geq 0,\qquad \intO{\frac{g_\ep(x,t)}{p_\ep(x)}}=1.
}
Next, notice that the estimates from Lemma \eqref{lemma:moments_bounds} are also independent of $\ep$. 
This is because there exists a $\beta_0>0$ such that $\beta_\ep(0)>\beta_0$ and $c_p$ in \eqref{cond:p1_ep} does not depend on $\ep$. We therefore have
\eq{d_\ep(t)\geq C \beta_0 e^{-c_0t}>0,}
and 
\eq{
{\rm ess} \sup_{t\in(0,T)} \lr{\alpha_\ep(t)+|\alpha_\ep'(t)|}\leq C,
}
with $C$ independent of $\ep$. From this we immediately get  
$${\rm ess}\sup_{t\in(0,T)}\{|c_\ep(t)|,|d_\ep(t)|, |d_\ep'(t)|\}<C,$$ and so the estimates \eqref{est:g^2_2}, \eqref{est:g_t_2}, \eqref{est:g_xx_2} are  all uniform with respect to $\ep$, we have
% \begin{equation}\label{est:g^2_uni}
% \sup_{t\in[0,T]} \int_\R \frac{g_\e^2}{p_\e}\,dx{}+{}C \beta_0 e^{-c_0T}\int_0^T \int_\R|\partial_x g_\e|^2\,dxdt{}\leq{}\int_\R \frac{g_0^2}{p_\e}\,dx.
% \end{equation}
% \red{add $\ep$ everywhere and convergence in the coefficients}
% Moreover, we can easily verify that the uniform in $\e$ estimates for $|\alpha(t)|$ and $\alpha'(t)$ from Lemma \ref{lemma:moments_bounds} can be proven the same way, and, so $\sup_{t\in(0,T)}\{|c(t)|,|d(t)|, |d'(t)|\}<C$. This allows us to show that the inequalities \eqref{est:g^2(1+x^2)_2} and \eqref{est:g_t} are also independent of $\e$ and they imply the uniform in $\e$ estimates as long as the data from the corresponding data integrals remain bounded.  
\begin{equation}
\label{est:g^2_3}
\mbox{\rm ess}\sup_{t\in[0,T]} \intO{\frac{e^{4|x|}g_\ep^2}{p_\e}}{}+C\beta(0)e^{-c_0T}\intTO{e^{4|x|}|\partial_x g_\ep|^2}{}\leq{}C\left\|\frac{e^{2|x|}g_0}{\sqrt{p_\ep}}\right\|^2_{L^2(\mathbb{R})},
\end{equation}
\eqh{
        \intTO{\frac{e^{4|x|}|\partial_tg_\ep|^2}{p_\e}}{}+{}C\beta(0)e^{-c_0T}\mbox{\rm ess}\sup_{t\in[0,T]}\intO{e^{4|x|}|\partial_x g_\ep(x,t)|^2}{}\\\leq{} C \|e^{2|x|}\partial_xg_0\|^2_{L^2(\mathbb{R})},
}
\begin{equation}
\label{est:g_xx_3}
    C\beta(0)e^{-c_0T}\intTO{ p_\e e^{4|x|}|\partial_{xx}^2g_\ep|^2}{}\leq{}C\left(\|e^{2|x|}\partial_xg_0\|^2_{L^2(\mathbb{R})}{}+{}\left\|\frac{e^{2|x|}g_0}{\sqrt{p_\ep}}\right\|^2_{L^2(\mathbb{R})}\right).
\end{equation}
Because $\frac{1}{p_\e}\geq 1$ from these inequalities we can actually deduce that the sequence $\{g_\e\}_{\e>0}$ is uniformly bounded in 
$L^\infty(0,T;  H^1(\R))$. Moreover, $\{\partial_t g_\e\}_{\e>0}$ is uniformly bounded in $L^2(0,T;  L^2(\R))$. 

By Aubin-Lions lemma this implies that $\{g_\e\}_{\e>0}$ is compact in  
$C([0,T];L^2(K))$
for any $K\subset \R$ compact. 
Moreover, estimate \eqref{est:g_xx_3} provides that
$\{g_\e\}_{\e>0}$ is uniformly bounded in $L^\infty(0,T;  H^2_{loc}(\R))$, and so, it has a *-weak limit in this space. 

Finally, of course, we also have $p_\ep(x)\to p(x)$ for all $x\in\R$.

We can therefore pass to the limit $\ep\to0$ in all terms of equation \eqref{eq:g1_ep} and the limit equation
\begin{equation*}
\partial_t g + \frac{(M-1)}{\sqrt{\tau}}a(t)p(x)\partial_x g - \frac{(M-1)^2}{2}\left(a^2(t)+\frac{1}{M-1}b(t)\right)p(x)\partial^2_{xx}g{}={}0,
\end{equation*}
is satisfied for a.e. $(x,t)\in (0,T)\times\R$.

Using a lower semi-continuity of functionals appearing in \eqref{est:g^2_3}--\eqref{est:g_xx_3} with respect to weak topologies, we pass to the limit $\ep\to0$ and obtain statements \eqref{est:non_lin_inf} and \eqref{est:non_lin_exp}.
Uniqueness of solution follows by the same argument as in the proof of  Theorem \ref{th:fixed_point}.
The proof of Theorem \ref{thm:main1} is complete. $\Box$

\section{Long-time behavior}
\label{sec:ltb}
The goal of this section is to show that asymptotic behavior of solutions of the PDE model \eqref{eq:f2} is consistent with the states corresponding to sorting and aggregate learning of interacting agents given by \eqref{def:learning f} and \eqref{def:sorting f}, respectively.

We prove the following
\begin{theorem}
\label{thm:long time}
{Let $M>1$, $\tau>0$, $\kappa\in(0,1)$, $M_c\in (1,M)$.}\\
 { Suppose that the probability function $p(x)$ is such that for some $c_p>0,$ conditions \eqref{cond:p_apriori} are satisfied, and moreover
  \begin{equation}
  \label{hyp:p}
  % 0<p'(x)\leq c_p(1-p(x)),\quad |p''(x)|\leq c_p(1-p(x)),\quad 
  |p''(x)|\leq c_p p'(x).
  \end{equation}}
Let  $g$ be a strong solution of the problem \eqref{eq:g1}-- \eqref{ini:g0} given in Theorem \ref{thm:main1}.\\

Then the following long-time asymptotic limits hold.
\begin{enumerate}
\item[i)] For  $f(x,t):=\frac{g(x,t)}{p(x)}$, we have 
\begin{equation}
\intO{p(x)^{\frac{7}{4}}(1-p(x))^{\frac{1}{4}}f^{\frac{7}{4} }(x,t)} \rightarrow 0,\quad \text{as}\quad t\to+\infty.
\end{equation}
In particular, the strong solution of the system \eqref{eq:f2} satisfies the asymptotic sorting condition  \eqref{def:sorting f}.
\item[ii)] In addition, let $0<p_{min}<\kappa$, 
and  $\tau$ be sufficiently small so that
\begin{equation}
\label{transport_vs_diffusion}
    \hat{c}_p\sqrt{\tau}(M-1)<p_{min}^4,\quad for \quad \hat{c}_p{}={}2C_0(1+c_p),
\end{equation}
and $C_0>1$  defined in \eqref{def_C0}, below.
% \red{Here it looks like we need $\tau$ small enough already}
% \blue{RIght. The first time, to get an estimate on lim sup a(t), and the second time make this limit small enough. }
% where $C_0>1$ is a certain numerical value, 

Then
% \red{maybe what specifically} such that
\[
\limsup_{t\to+\infty} |a(t)| \leq \frac{\hat{c}_p}{p_{min}^4}\sqrt{\tau}.
\]
In particular, if $\tau$ is further restricted so that 
\eq{\kappa \pm \frac{\hat{c}_p}{p_{min}^4}\sqrt{\tau}{}\in{}(\frac{M_c-1}{M},\frac{M_c}{M}),} then asymptotically, $f$ verifies the aggregate learning condition \eqref{def:learning f}.
\end{enumerate}

\end{theorem}

\begin{remark}
    If the function $p(x)$ is strictly increasing, conditions \eqref{hyp:p} only restrict the values of the function asymptotically, when $|x|\to+\infty.$ For example, regular functions $p$ that equal $\frac{x^\alpha}{1+x^\alpha},$ ($\alpha>0$), for large positive $x,$ and equal $\frac{1}{1+|x|^\alpha},$ for large negative $x,$ verify conditions \eqref{hyp:p}.
\end{remark}
\begin{remark}
    Condition \eqref{transport_vs_diffusion} appears in the proof of this theorem as the measure of the ``strength'' of the diffusion relative to the transport coefficient in \eqref{eq:f2}. Informally, it restricts the model to the transport dominated regime. 
\end{remark}

The proof of Theorem \ref{thm:long time} is given below in a series of lemmas.

%We will need the following equations for  functions $a(t)$ and $\beta(t)$
%that are derived in the way similar to the corresponding equations for the linearized equation in section \ref{sec:apriori estimates}, under the conditions
%\eqref{cond:p_apriori} of function $p(x).$

%The equation for $a(t)$ is
%\begin{equation}
%    \label{eq:a_2}
%    a'{}(t)={}-\frac{a(t)(M-1)}{\sqrt{\tau}}\intO{pp'f}
%    {}+{}\frac{(M-1)^2}{2}\left(a^2(t)+\frac{b(t)}{M-1}\right)\intO{pp''f},
%\end{equation}
%and the equation for $\beta(t):$
%\begin{multline}
%    \label{eq:beta_2}
%    \beta'(t){}={}
%    \frac{a(t)(M-1)}{\sqrt{\tau}} \intO{p((1-p)p)'f}\\
%    {}+{}
%    \frac{(M-1)^2}{2}\left(a^2(t)+\frac{1}{M-1}b(t)\right)\intO{p((1-p)p)'' f}.
%\end{multline}

Our first long-time asymptotic result demonstrates that diffusion drives all mass towards extreme values 
$x=\pm\infty.$ This is established in the following lemma.
\begin{lemma}
\label{lim beta(t)} Let assumptions of Theorem \ref{thm:long time} part (i) be satisfied.
Then, we have
\begin{equation}
\intO{p(x)^{\frac{7}{4}}(1-p(x))^{\frac{1}{4}}f^{\frac{7}{4} }(x,t)} \rightarrow 0,\quad \text{as}\quad t\to+\infty.
\end{equation}
%If, in addition,
%\begin{equation}
%    \label{integral_p(1-p)}
%    \int (p(x)-p_{min})^2(1-p(x))^4\,dx<+\infty.
%\end{equation}
%Then, 
%\begin{equation}
%\label{moment:(p-p_min)(1-p)}
%\lim_{t\to+\infty}
%\int (p-p_{min})(1-p)f(x,t)\,dx{}={}0.
%\end{equation}
%\red{How is it deduced?}
%\blue{resolved}
\end{lemma}

\begin{proof}
Let $g$ be a solution of \eqref{eq:g1} given in Theorem \eqref{thm:main1}. Given the regularity of the solution $g$, the integral $\intO{g^2(x,t)}$ is continuous and  weakly differentiable function, with
\[
\Dt\intO{g^2}{}={}2\intO{g\partial_tg}\in L^2((0,T)),\quad \forall T>0.
\]

Multiplying equation \eqref{eq:g1} by $g$ and integrating by parts, we get 
\begin{equation*}
         \frac{1}{2}\Dt\intO{g^2}+ \frac{c(t)}{2} \intO{p(x)  \partial_x
g^2}  + \frac{d(t)}{2} \intO{p'(x)  \partial_x g^2}+ d(t) \intO{p(x)  (\partial  _x g)^2}=0.
\end{equation*}
Integrating by parts again, we get 
\begin{equation}
         \frac{1}{2}\Dt\intO{g^2}- \frac{c(t)}{2} \intO{p'(x)  
g^2}  - \frac{d(t)}{2} \intO{p''(x)   g^2 }+ d(t) \intO{p(x)  (\partial  _x g)^2}=0.
\end{equation}
Hence,
\begin{equation}
         \Dt\intO{g^2} \leq  c(t) \intO{p'(x)  
g^2}  + d(t) \intO{p''(x)   g^2}.
\end{equation}
From this, we obtain inequality 
\begin{equation}
\label{e2}
         \left(\Dt\intO{g^2}\right)_+ \leq  C\left(|c(t)| + d(t)\right) \intO{g^2},
\end{equation}
where $C$ is the maximum of $p'$ and $p''.$

Multiplying equation \eqref{eq:g1} by $g/p$ and integrating by parts we obtain estimate
\begin{equation}
 \label{e3}
{\rm ess}\sup_{t\in[0,T]} \intO{\frac{g^2}{p(x)}}{}+{}2\int_0^Td(t)\intO{(\partial_x g)^2}\,\dt{}\leq{}\intO{\frac{g_0^2}{p}}.
\end{equation}

Using the Nash inequality
\[
\|g(\cdot,t)\|_{L^2(\R)}^6{}\leq C\|g(\cdot,t)\|_{L^1(\R)}^4\|\partial_x g(\cdot,t)\|_{L^2(\R)}^2,\quad \forall t>0
\]
we get
\begin{multline*}
\int_0^T d(t)\|g(\cdot,t)\|_{L^2(\R)}^6{}\,\dt\leq{} C\int_0^T  d(t) \|g(\cdot,t)\|_{L^1(\R)}^4\|\partial_x g(\cdot,t)\|_{L^2(\R)}^2 \ \dt
\\
\leq
C{\rm ess}\sup_{t\in[0,T]}\|g(\cdot,t)\|_{L^1(\R)}^4\int_0^T  d(t) \|\partial_x g(\cdot,t)\|_{L^2(\R)}^2 \ \dt.
\end{multline*}
Since $\intO{|g|} \leq  \intO{|f|}\leq 1,$ and using estimate \eqref{e3}, we can find $C,$ independent of $T,$ such that
\[
\int_0^T d(t)\|g(\cdot,t)\|_{L^2(\R)}^6\ \dt \leq C,
\]
meaning that 
\[d(t) \left(\intO{g^2}\right)^3 \in L^1(0,\infty).
\]

Since $\beta(t)$ is uniformly bounded and $\beta(t)\leq b(t)\leq Cd(t),$ we have
\begin{equation}
\phi(t):=\beta(t) \left(\intO{g^2}\right)^3 \in L^1(0,\infty).
\end{equation}

Moreover, we can compute
\begin{equation}
        \phi'(t)= \beta'(t) \left(\intO{g^2}\right)^{3}   + 3 \beta(t) \left(\intO{g^2}\right)^{2} \left( \Dt \intO{g^2} \right). 
\end{equation}

Recall that  \ref{est:alpha'} provides that for $t\geq0,$ we have $|\beta'(t)|{}\leq C\beta(t)$.
Using this and \eqref{e2},
we find that
\eqh{
\left( \phi'(t)\right)_+\leq C \beta(t) \left(\intO{g^2}\right)^{3} +  3C \beta(t)  \left(|c(t)| + d(t)\right) \left(\intO{g^2}\right)^{3}\\
{}\leq{} C\beta(t)\left(\intO{g^2}\right)^{3}
{}\leq{}C \phi(t)\in L^1(0,\infty). 
}

Let $t_1<t_2$ and $\phi(t_2)\geq \phi(t_1).$ By the fundamental theorem of calculus 
\begin{equation}
\label{e5}
\phi(t_2)-\phi(t_1)= \int_{t_1}^{t_2}\phi'(t)\,\dt \leq \int_{t_1}^{t_2} (\phi'(t))_+\,\dt.
\end{equation}
Because $\phi\in L^1(0,\infty),$ there is a sequence of points $t_n\to\infty,$ such that $\phi(t_n)\to0.$ If $\phi(t)$ does not converge to $0,$ there must be $\delta>0$ and 
a sequence $s_n\to\infty,$ for which $\phi(s_n)>\delta.$ From these two sequences we can select monotone subsequences $t_{n_k}$ and $s_{n_k}$ such that
\[
t_{n_{k-1}}<s_{n_{k-1}}{}< {}t_{n_k}{}<{}s_{n_k}
\]
and
\[
\phi(s_{n_k})-\phi(t_{n_k})>\delta/2,\quad
\]
for all $k\geq 0.$
This leads to a contradiction, since from \eqref{e5}
\[
\sum_{k=1}^K (\phi(s_{n_k})-\phi(t_{n_k})){}\leq {}\int_0^\infty (\phi'(t))_+\,\dt<+\infty,
\]
for any integer $K,$ and, at the same time,
\[
\sum_{k=1}^K (\phi(s_{n_k})-\phi(t_{n_k})){}\geq{}\frac{\delta K}{2}.
\]
Thus, $\lim_{t\to+\infty}\phi(t)=0.$ 

 By H\"older's inequality, we have
\begin{multline*}
\intO{p(x)^{\frac{7}{4}}(1-p(x))^{\frac{1}{4}}f^{\frac{7}{4} }}\leq \left(\intO{p(x) (1-p(x))f } \right)^\frac{1}{4} \left(\intO{p^2(x)f^2} \right)^{\frac{3}{4}}\\
=\beta(t)^{\frac{1}{4}}\left(\intO{g^2} \right)^\frac{3}{4}{}={}\phi^\frac{3}{4}{}(t).
\end{multline*}
Therefore,
\begin{equation*}
\intO{p(x)^{\frac{7}{4}}(1-p(x))^{\frac{1}{4}}f^{\frac{7}{4} }} \rightarrow 0
\end{equation*}
as $t\rightarrow \infty.$ 
This proves the statement of the lemma.

%Using the three-point H\"older's inequality with exponents $p_1=\frac{12+\gamma}{3+\gamma},p_2=\frac{12+\gamma}{3},$ and $p_3=\frac{12+\gamma}{6},$ we have

%\begin{multline*}
%\int (p(x)-p_{min})(1-p(x))f dx\\
%= \int (p(x)-p_{min})^{\frac{6+\gamma}{12+\gamma}}(1-p(x))^{\frac{\gamma}{12+\gamma}} f^{\frac{6+\gamma}{12+\gamma}} [(p(x)-p_{min})(1-p(x))^2]^{\frac{6}{12+\gamma}}f^{\frac{6}{12+\gamma}}\,dx
%\\
%\leq \left((\int p(x)-p_{min})^{\frac{6+\gamma}{3+\gamma}}(1-p(x))^{\frac{\gamma}{3+\gamma}}f^{\frac{6+\gamma}{3+\gamma}}dx \right)^{\frac{1}{p_1}}
%\left(\int (p(x)-p_{min})^2 (1-p(x))^4 dx\right)^{\frac{1}{p_2}}
%\left(\int f\,dx \right)^{\frac{1}{p_3}}
%\\
%\leq C \left(\int p(x)^{\frac{6+\gamma}{3+\gamma}}(1-p(x))^{\frac{\gamma}{3+\gamma}}f^{\frac{6+\gamma}{3+\gamma}}\,dx \right)^{\frac{1}{p_1}}, 
%\end{multline*}
%where the last inequality is true because of 
%\eqref{integral_p(1-p)}  and $\int f dx \leq 1.$
%This proves the second statement of the lemma. 

\end{proof}

\begin{corollary}
    The last lemma implies  that, asymptotically, the sorting phenomenon, defined by \eqref{def:sorting f}, holds.
\end{corollary}
\begin{proof}
    Indeed, for any $R>0,$ since $\min_{x\in[-R,R]}p(x)>0$ and $\min_{x\in[-R,R]}(1-p(x))>0,$ we can estimate
\eqh{
        \int_{-R}^R f\,\dx &\leq
        (2R)^{\frac{4}{7}}\left(\int_{-R}^R f^{\frac{7}{4}}\,\dx
        \right)^{\frac{4}{7}}
        \\
        &\leq (2R)^{\frac{4}{7}}\left(\int_{-R}^R \left(\frac{p(x)}{\min_{[-R,R]}p}\right)^{\frac{7}{4}} 
        \left(\frac{1-p(x)}{\min_{[-R,R]}(1-p)}\right)^{\frac{1}{4}}
        f^{\frac{7}{4}}\,\dx
        \right)^{\frac{4}{7}}\\
        &\leq
        \frac{(2R)^{\frac{4}{7}}}{\min_{[-R,R]}p\min_{[-R,R]}(1-p)^\frac{1}{7}}\left(\int_{-R}^R p^{\frac{7}{4}} 
        (1-p)^{\frac{1}{4}}
        f^{\frac{7}{4}}\,\dx
        \right)^{\frac{4}{7}}\to 0,
}
    as $t\to+\infty.$
\end{proof}

Our next result is dedicated to the asymptotic behaviour of $a(t)$, defined in \eqref{def:a}. We have:

\begin{lemma}
\label{lemma:lim_a(t)}
If
\[
\int_0^\infty\!\!\!\intO{p(x)p'(x)f(x,t)}\,\dt{}<{}+\infty,
\] 
then,  $\lim_{t\to\infty}a(t)$ exists. 
If 
\[
\int_0^\infty\!\!\!\intO{p(x)p'(x)f(x,t)}\,\dt{}={}+\infty,
\] 
then
\[
\limsup_{t\to+\infty}|a(t)|{}\leq{} c_p\sqrt{\tau}.
\]
\end{lemma}
\begin{proof}
The equation for $a'(t)$ is obtained by multiplication of equation \eqref{eq:f2} by $\kappa-p(x)$ and integration by parts, we have
\begin{equation*}
   a'{}(t)={}-\frac{a(t)(M-1)}{\sqrt{\tau}}\intO{pp'f}
   {}+{}\frac{(M-1)^2}{2}\left(a^2(t)+\frac{b(t)}{M-1}\right)\intO{pp''f}.
\end{equation*}
We write this equation as
    \begin{equation}
    a'{}(t)={}-\omega(t)a(t)  
    -\frac{M-1}{2}b(t)\intO{p(x)p''(x)f(x,t)},
    \end{equation}
     where we set 
     \[
     \omega(t)=\frac{(M-1)}{\sqrt{\tau}} \intO{p(x)\left(p'(x) +  \frac{(M-1)\sqrt{\tau}a(t)}{2}p''(x)\right)f(x,t)}.
     \]
    Since $|a(t)|<1,$ using conditions \eqref{cond:p_apriori}, \eqref{hyp:p} and \eqref{transport_vs_diffusion} (which implies that $\sqrt{\tau}(M-1)c_p<1$),  we have 
\begin{equation}
\label{est:omega}
0\leq\frac{(M-1)}{2\sqrt{\tau}} \intO{p(x)p'(x)f(x,t)}\leq \omega(t) \leq \frac{3(M-1)}{2\sqrt{\tau}} \intO{p(x)p'(x)f(x,t)}\geq0.
\end{equation}
So the integrals $\int_0^\infty\omega\,\dt$ and $\int_0^\infty\!\!\! \intO{pp'f}\,\dt$ converge or diverge at the same time.    
We solve for $a(t)$:
    \begin{equation}
    \label{formula:a(t)}
    a(t) = a(0)e^{-\int_0^t\omega(s)\,{\rm d}s}{}-\frac{M-1}{2}\int_0^t e^{-\int_\zeta^t \omega(s)\,{\rm d}s}b(\zeta)\intO{p(x)p''(x)f(x,\zeta)}\,{\rm d}\zeta.
    \end{equation}
We consider two cases $\int_0^\infty \omega(t)\,\dt{}<{}\infty$ and $\int_0^\infty \omega(t)\,\dt{}={}\infty.$

\medskip

Suppose first that $\int_0^\infty \omega(t)\,\dt{}<{}\infty.$ Note that the first term in \eqref{formula:a(t)} $a(0)e^{-\int_0^t\omega(s)\,{\rm d}s}{}$ has a limit at $t\to+\infty,$ because $\omega(t)\geq0.$ 
    
 Next, we will check the limit as $t\to\infty$ of the the function
    \[
    r(t) {}={}\int_0^t e^{-\int_\zeta^t \omega(s)\,\ds}b(\zeta)\intO{p(x)p''(x)f(x,\zeta)}\,\dzeta.
    \]
    We have
    \begin{multline*}
    r(t_2)-r(t_1) = \int_{t_1}^{t_2} e^{-\int_\zeta^{t_2} \omega(s)\,\ds}b(\zeta)\intO{p(x)p''(x)f(x,\zeta)}\,\dzeta\\
    {}+{}\int_0^{t_1} \left( e^{-\int_\zeta^{t_2} \omega(s)\,\ds}-
    e^{-\int_\zeta^{t_1} \omega(s)\,\ds} \right)b(\zeta)\intO{p(x)p''(x)f(x,\zeta)}\,\dzeta.
    \end{multline*}
    Due to assumptions \eqref{hyp:p} and estimate \eqref{est:omega}, there is $C$ independent of $\zeta>0$ for which
    \begin{equation*}
     \left|b(\zeta)\intO{p(x)p''(x)f(x,\zeta)}\right|\leq C\omega(\zeta).   
    \end{equation*}
Therefore,  we can estimate
    \begin{equation*}
       |r(t_2)-r(t_1)|{}\leq{} C\lr{1-e^{-\int_{t_1}^{t_2}\omega(s)\,{\rm d}s}} {}+{}
       C\lr{1-e^{-\int_{t_1}^{t_2}\omega(s)\,{\rm d}s}}
       \int_0^\infty \omega(t)\,\dt
\to 0,
    \end{equation*}
    as $t_2,t_1\to+\infty$.
    Thus, $\lim_{t\to+\infty} r(t)$ exists and so does $\lim_{t\to+\infty} a(t)$.

\medskip 

Suppose now that $
\int_0^\infty \omega(t)\,\dt{}={}\infty$. 
From \eqref{formula:a(t)}, we find that
\[
|a(t)|{}\leq{} |a(0)|e^{-\int_0^t\omega(s)\,\ds}
{}+{}c_p\sqrt{\tau}\int_0^t e^{-\int_\zeta^t\omega(s)\,\ds}\omega(\zeta)\,\dzeta{}\leq
|a(0)|e^{-\int_0^t\omega(s)\,\ds}
{}+{}c_p\sqrt{\tau}.
\]
It follows that 
\[
\limsup_{t\to+\infty}|a(t)|\leq c_p\sqrt{\tau}.
\]    
\end{proof}

%\textcolor{red}{Remark: I do not think $p_{min}$ is needed in the above inequality. In fact,
%\[
%\sup_{t\in[0,T]}\int e^{|x|}f(x,t),\,dx{}=\sup_{t\in[0,T]}\int e^{2|x|-|x|}\frac{\sqrt{p(x)}}{\sqrt{p(x)}}f(x,t)\,dx{}
%\]
%\[
%\leq\sup_{t\in[0,T]}\left(\int e^{4|x|}p(x)f^2(x,t)\,dx\right)^{1/2}\left(\int \frac{e^{-2|x|}}{p(x)}\,dx\right)^{1/2}.
%\]
%Recall that
%\[
%p(x)=\begin{cases}
%1-\frac{1}{2(1+x)} & \mathrm{if}\;\;x>0\\
%\frac{1}{2(1-x)} & \mathrm{if}\;\;x<0.
%\end{cases}
%\]
%Hence, 
%\[
%\int \frac{e^{-2|x|}}{p(x)}\,dx=\int_{-\infty}^0e^{2x}2(1-x)dx+\int_0^{\infty}e^{-2x}\frac{2(1+x)}{1+2x}dx\approx 2.2982.
%\]
%}

Before concluding the proof of Theorem \ref{thm:long time} let us summarize some of the properties of the unique regular solution of \eqref{eq:f2}.

\begin{lemma}\label{lemma:exponential_weight_f}
        Let \eqref{g0} hold for $g_0=p(x)f_0$, then, for all $T>0,$ there is $C_T,$ that depends on $(T,M,\tau,p,f_0),$ such that the solution of system \eqref{eq:f2} satisfies
    \[
   \mbox{\rm ess} \sup_{t\in[0,T]}\intO{e^{4|x|}p(x)f^2(x,t)}
    {}+\intTO{e^{4|x|}|\partial_{x}(pf)|^2}
    \leq C_T.
    \]
Moreover,  if $p_{min}>0$, then 
\eq{ \mbox{\rm ess}  \sup_{t\in[0,T]}\intO{e^{|x|}f(x,t)}\leq C_T.}
\end{lemma}

\begin{proof}
The first estimate follows by substituting $g=p(x)f(x,t)$ in \eqref{est:non_lin_inf}.
Using this estimate we find that
\begin{multline}
\label{est:expon}
 \mbox{\rm ess}  \sup_{t\in[0,T]}\intO{ e^{|x|}f(x,t)}{}={}
 \mbox{\rm ess} \sup_{t\in[0,T]}\intO{e^{2|x|-|x|}f(x,t)}{}\\
\leq{}
 \mbox{\rm ess}  \sup_{t\in[0,T]}\left(\intO{e^{4|x|}f^2(x,t)}\right)^{1/2}\left(\intO{e^{-2|x|}}\right)^{1/2}\\
\leq{}
 \mbox{\rm ess} \sup_{t\in[0,T]}\frac{1}{\sqrt{p_{min}}}\left(\intO{e^{4|x|}p(x)f^2(x,t)}\right)^{1/2}\left(\intO{e^{-2|x|}}\right)^{1/2}
{}\leq{}C_T
\end{multline}
for some other $C_T.$
\end{proof}

We now prove the last part of Theorem \ref{thm:long time} formulated in the following lemma.

\begin{lemma} There is a numerical value $C_0$ such that if \eqref{transport_vs_diffusion} holds, then
\begin{equation} 
\label{est:lim_a} 
\limsup_{t\to+\infty}|a(t)| \leq \frac{\hat{c}_p}{p_{min}^4}\sqrt{\tau}.
\end{equation}
\end{lemma}

\begin{proof}
By Lemma \ref{lemma:lim_a(t)}, it suffices to consider the case when $\lim_{t\to\infty} a(t)$ exists. Moreover, we can assume that this limit is nonzero, since otherwise \eqref{est:lim_a} follows trivially.

The proof uses a contradiction argument and is divided into several steps.

\medskip

\noindent{\emph{Step 1.}} Setup of the contradiction argument.\\
After fixing $C_0$, $\hat{c_p}$ and $\tau$ so that condition \eqref{transport_vs_diffusion} is satisfied, we will show that assumption
\begin{equation}
\label{a is large}
\lim_{t\to\infty} a(t) > \frac{\hat{c}_p}{p_{min}^4}\sqrt{\tau},\quad \left(\mbox{\rm or}\,
\lim_{t\to\infty}a(t)<-\frac{\hat{c}_p}{p_{min}^4}\sqrt{\tau}\right)
\end{equation}
necessarily implies that 
\begin{equation}
\label{x<x_0}
\lim_{t\to\infty}\int_{x<x_0}f(x,t)\,\dx{}={}0,\quad
\left(\mbox{\rm or}\,\lim_{t\to\infty}\int_{x>x_0}f(x,t)\,\dx{}={}0\right),
\end{equation}
for any $x_0.$ As we shall see below, the first inequality in \eqref{a is large} and the first limit in \eqref{x<x_0} then imply $\lim_{t \to \infty} a(t)<0,$ which is a contradiction.

Indeed, it is enough to take  $x_0$ in \eqref{x<x_0} large enough, so that $p(x_0)>\kappa$
and write 
\begin{multline*}
a(t)=\kappa-\intO{pf}{}={}\kappa- \int_{x<x_0}pf\,\dx- \int_{x\geq x_0}pf\,\dx{}\leq{}\kappa- \int_{x<x_0}pf\,\dx-p(x_0)\int_{x\geq x_0}f\,\dx\\
{}={}
\kappa- \int_{x<x_0}pf\,\dx-p(x_0)\left(1-\int_{x<x_0}f\,\dx\right)\\
{}={}\kappa - p(x_0)+ p(x_0)\int_{x<x_0}f\,\dx-\int_{x<x_0}pf\dx {}\to{} \kappa-p(x_0)<0,
\end{multline*}
where we used
\[
\int_{x<x_0}f\,\dx{}+{}\int_{x>x_0}f\,\dx=1.
\]
This implies $\lim_{t \to \infty} a(t)<0,$ contradicting the assumptions.
Similarly, one gets a contradiction starting with the second inequality in \eqref{a is large} and \eqref{x<x_0}. For if $p_{min}<\kappa,$ there is $x_0$ such that $p(x_0)<\kappa,$ and we have
\begin{multline*}
a(t)=\kappa-\intO{pf}{}={}\kappa- \int_{x<x_0}pf\,\dx- \int_{x\geq x_0}pf\,\dx{}\geq{}\kappa- p(x_0)\int_{x<x_0}f\,\dx-\int_{x\geq x_0}pf\,\dx\\
{}={}
\kappa-p(x_0)\left(1-\int_{x\geq x_0}f\,\dx\right)- \int_{x\geq x_0}pf\,\dx
{}\to {}\kappa - p(x_0)>0,
\end{multline*}
contradicting the fact that $\lim_{t\to+\infty} a(t)<0.$

Because, as we have already shown, $\int_{[-R,R]}f(x,t)\,\dx\to0,$ for any $R>0,$ it suffices to consider only one value of $x_0.$ It will be chosen later.

\medskip

\noindent{\emph{Step 2.}} Rescaling of the equation and of the test function.\\ 
% Proof of \eqref{a is large}$\implies$ \eqref{x<x_0}.\\
We rescale the time $t$ and change it to a new variable $s$ defined by $t=\frac{\sqrt{\tau}}{M-1}s.$ Let $h(x,s){}={}f(x,t)$, then $h$ solves equation
\begin{equation}
    \label{eq:h}
    \partial_s h{}+{}\bar{a}(s)\partial_x(ph){}-{}\frac{\sqrt{\tau}(M-1)}{2}\left(\bar{a}^2(s){}+{}\frac{\bar{b}(s)}{M-1}\right)\partial_{xx}^2 (ph){}={}0,
\end{equation}
where we set $\bar{a}(s){}={}a\left(\frac{\sqrt{\tau}}{M-1}s\right)$ and $\bar{b}(s){}={}b\left(\frac{\sqrt{\tau}}{M-1}s\right).$ We will also denote
\begin{equation}
\label{def:bar d}
\bar{d}(s){}={}\frac{\sqrt{\tau}(M-1)}{2}\left(\bar{a}^2(s){}+{}\frac{\bar{b}(s)}{M-1}\right).
\end{equation}

First,  observe that because $\bar{d}(s)>0$ for $s\in[0,S]$, the function $h(x,s)$ verifies an inequality similar to one in Lemma \ref{lemma:exponential_weight_f} on every interval $s\in[0,S]:$  
\begin{equation}
\label{eq:expon h^2}
{\rm ess}    \sup_{s\in[0,S]}\intO{e^{4|x|}h^2(x,s)}
    {}+{}
    \int_0^S\!\!\!\intO{e^{4|x|}|\partial_{x}(p(x)h(x,s))|^2}\,{\rm d}s
    <C_S,
\end{equation}
for some $C_S$ that depends on $(S,M,\tau,p,f_0).$

In particular, the same way we deduced \eqref{est:expon}, we show that $h(x,s)$ is integrable in $x$ with weight $e^{|x|},$
for every $s\geq0$ and there is a bound
\begin{equation}
    \label{eq:expon h}
    \sup_{s\in[0,S]}\intO{e^{|x|}h(x,s)}{}\leq{}C_S,
\end{equation}
for some $C_S.$

We now look for a test function $\phi$ that will allow us to get an analogue of estimate \eqref{eq:expon h} in the reference system  moving with the transport velocity $\bar a(s) p(x)$. To this purpose consider an even, $C^2(\mathbb{R})$ function $\psi$ with values
\[
\psi(x){}={}\left\{
\begin{array}{lc}
     e^{x}& x\geq 2  \\
     1& 0\leq x\leq 1
\end{array}
\right.
\]
%Let $\beta>0$ and consider function 
%\[
%\psi(x){}={}\left\{
%\begin{array}{lc}
%     \dfrac{(2-x)^{\beta}}{2^{\beta-1}}& x\leq 0  \\
%     1& x>1
%\end{array}
%\right.
%\]
and smoothly and monotonically interpolated on the intervals $[-2,-1]$ and $[1,2].$ Let
\eq{\label{def_C0}
C_0=\max\left\{\sup_x|\partial_x\log \psi(x)|,\sup_x|\partial_x^2\log\psi(x)|\right\}<+\infty.
}
The constant $C_0$ together with $c_p$ given in \eqref{cond:p_apriori} and \eqref{hyp:p} can be now used to determine $\hat{c}_p$.
Let $\phi$ be the solution of the transport equation 
\begin{equation}
\label{eq:transport}
\left\{
\begin{array}{l}
\partial_s\phi(x,s){}+{} \bar{a}(s)p(x)\partial_x\phi(x,s){}={}0,\\
\phi(x,s_0)=\psi(x).
\end{array}\right.
\end{equation}
We can show that that for any $s>s_0,$ $\phi(x,s)$ is dominated by exponential function $e^{|x|}.$ 
Indeed, let $X(x_0,s)$ be a trajectory determined by the ODE:
\begin{equation}\label{X:ODE}
\left\{
\begin{array}{l}
\dfrac{dX}{ds}=\bar{a}(s)p(X), \\
X(x_0,s_0){}={}x_0,
\end{array}
    \right.
\end{equation}
for some $s_0$ large enough to be determined later on.
Thus it holds that:
\[
|X(x_0,s)-x_0|\leq\int_{s_0}^s\sup_{t\in(s_0,s)}(|\bar{a}(t)p(X(x_0,t))|\,\dt\leq s-s_0,\quad \forall x_0\in\R.
\]
Because $X(\cdot,s)$ is a diffeomorphism $\mathbb{R}\to\mathbb{R},$  denoting by $Y(\cdot,s)$ its inverse, we also have
\[
|x-Y(x,s)|\leq s-s_0,\quad \forall x\in\R.
\]
Since $\phi(x,s)=\psi(Y(x,s))$ and $|\psi(x)|\leq{} 2e^{|x|},$ then 
\begin{equation}
\label{eq:exp bound on phi}
|\phi(x,s)|\leq 2e^{s-s_0}e^{|x|},
\end{equation}
for all $s>s_0.$ This means, due to \eqref{eq:expon h}, that  $\phi$ has a suitable growth to be a test function in the equation \eqref{eq:h}.

\medskip

\noindent{\emph{Step 3.}} Derivation of the equation for $\intO{h\phi}.$

Let $\omega_L(x)$ be a smooth cutoff function for interval $[-L,L],$
equal $1,$ for $x\in[-L,L]$ and $0,$ for $|x|>2L.$
Multiplying equation \eqref{eq:h} by $\phi\omega_L$
we get
\begin{multline*}
\partial_s(h\phi\omega_L) +  \bar{a}\partial_x(ph\phi\omega_L) {}-{}\bar{d}\partial_x(\phi\omega_L\partial_x(ph))
{}+{}\bar{d}\partial_x(ph\omega_L\partial_x\phi)
\\
{}={}
\bar{a}ph\phi\partial_x\omega_L
{}+{}
\bar{d}ph\omega_L\partial_{xx}^2\phi
{}-{}\bar{d}
\partial_x(ph)\phi\partial_x\omega_L+
\bar{d}ph\partial_x\phi\partial_{x}\omega_L.
\end{multline*}
Integrating this equation over $\mathbb{R}\times[s_0,s]$ we get
\eq{\label{eq:expon_h0}
    \intO{h\phi\omega_L\big|_{s_0}^s}{}=&{}
    \int_{s_0}^s\!\!\!\intO{\bar{a}(t)ph\phi\partial_x\omega_L}\,\dt
    {}+{} \int_{s_0}^s\!\!\!\intO{\bar{d}(t)ph\omega_L\partial_{xx}^2\phi}\,\dt\\
    &\quad {}-{}\int_{s_0}^s\!\!\!\intO{\bar{d}(t)\phi
\partial_x(ph)\partial_x\omega_L-
\bar{d}(t)ph\partial_x\phi\partial_{x}\omega_L}\,\dt\\
&=I_1+I_2+I_3+I_4.
}

Using the Dominated Convergence Theorem, we  can easily pass to the limit $L\to+\infty$ in the left-hand side of \eqref{eq:expon_h0}. Indeed, since $h$ is integrable with the exponential weight, see \eqref{eq:expon h}, and $\phi(x,s)$ is bounded by the exponential function, see \eqref{eq:exp bound on phi},  the limit of this term equals to $\intO{h\phi\big|_{s_0}^s}{}$.

We now pass to the limit $L\to+\infty$ in $I_1$. 
Using \eqref{eq:exp bound on phi}, \eqref{eq:expon h^2} and the definition of function $\omega_L,$ we can estimate it as
\begin{multline}
   |I_1|= \left| \int_{s_0}^s\!\!\!\intO{ \bar{a}ph\phi\partial_x\omega_L}\,\dt\right|
    {}\leq{}
C_s\int_{s_0}^s\left(\intO{e^{4|x|}h^2}\right)^{\frac{1}{2}}\left(\intO{e^{-2|x|}|\partial_x\omega_L|^2}\right)^{\frac{1}{2}}\,\dt\\
    {}\leq{}C_s
\int_{s_0}^s\left(\int_{|x|>L} e^{-2|x|}\,\dx\right)^{\frac{1}{2}}\dt
    \to0,
\end{multline}
as $L\to+\infty$. By the similar arguments, we can also show that $I_3\to0$.  

The remaining integrals, $I_2$ and $I_4$, contain derivatives of $\phi,$ which need to be estimated first. 
For  convenience we write
\[\px\phi=\phi\px\log\phi \quad \text{and} \quad
\partial_{xx}^2\phi{}={}\phi(\partial_{xx}^2\log\phi{}+{}(\partial_x\log\phi)^2).
\]
Assuming now that $\partial_x\log \phi$ and $\partial^2_{xx}\log\phi $ are bounded, we can pass to the limit $L\to+\infty$ in $I_2$ and $I_4$ to conclude that
\eqh{
\lim_{L\to+\infty} I_2= &\lim_{L\to+\infty} \int_{s_0}^s\!\!\intO{\bar{d}(t)ph\phi\omega_L(\partial_{xx}^2\log\phi{}+{}(\partial_x\log\phi)^2)}\,\dt\\
=&\int_{s_0}^s\!\!\intO{\bar{d}(t)ph\phi(\partial_{xx}^2\log\phi{}+{}(\partial_x\log\phi)^2)}\,\dt.
}
Thus, taking $L\to+\infty$ in \eqref{eq:expon_h0} implies that
\begin{equation}
\label{eq:exp moment h}
\intO{h\phi\big|_{s_0}^s}{}={}\int_{s_0}^s\!\!\intO{\bar{d}(t)ph\phi(\partial_{xx}^2\log\phi{}+{}(\partial_x\log\phi)^2)}\,\dt,
\end{equation}
provided that $\partial_x\log \phi$ and $\partial^2_{xx}\log\phi $ are bounded.
\medskip

\noindent{\emph{Step 4.}} Estimate of $\partial_x\log \phi$ and $\partial^2_{xx}\log\phi $.

Due to \eqref{eq:transport}, the functions $u=\partial_x\log \phi$ and $v=\partial^2_{xx}\log\phi,$ satisfy: 
\begin{equation}
\left\{
\begin{array}{l}
\partial_s u + \bar{a}p\partial_xu{}={}-\bar{a}p'u,\\
u(x,s_0)=u_0(x)=\px\log\psi(x),
\end{array}
\right.  
\end{equation}
and
\begin{equation}
\left\{
\begin{array}{l}
\partial_s v{}+{}\bar{a}p\partial_xv{}={}-2\bar{a}p'v-\bar{a}p''u,\\
v(x,s_0)=v_0(x)=\partial^2_{xx}\log\psi.
\end{array}
\right.  
\end{equation}

% \[
% \partial_s u + \bar{a}p\partial_xu{}={}-\bar{a}p'u,\quad u(x,s_0)=u_0(x)=\px\log\psi(x),
% \]
% \[
% \partial_s v{}+{}\bar{a}p\partial_xv{}={}-2\bar{a}p'v-\bar{a}p''u,\quad v(x,s_0)=v_0(x)=\partial^2_{xx}\log\psi.
% \]
Solving the equation for $u,$ we get
\begin{equation}
\label{est:u}
|u(X(x_0,s),s)|\leq|u_0(X(x_0,s_0))|e^{\int_{s_0}^s|\bar{a}(\zeta)|p'(X(x_0,\zeta))\,{\rm d}\zeta}{}\leq{}C_0 e^{\int_{s_0}^s |\bar{a}(\zeta)|p'(X(x_0,\zeta))\,{\rm d}\zeta}.
\end{equation}
Similarly, we solve for $v:$
% \eq{\label{140}
% &v(X(x_0,s),s)\\
% &\quad=e^{-2\int_{s_0}^s \bar{a}(\zeta) p'(X(x_0,\zeta )) {\rm d}\zeta}\Big(v_0(X(x_0,s_0))\\
%  &\hspace{4.7cm}-\int_{s_0}^se^{2\int_{s_0}^r \bar{a}(\zeta) p'(X(x_0,\zeta )) \dzeta} \bar{a}(r) p''(X(x_0,r))u(X(x_0,r),r) {\rm d}r \Big),
% }
\eqh{
&v(X(x_0,s),s)\\
&\quad=e^{-2\int_{s_0}^s \bar{a}(\zeta) p'(X(x_0,\zeta )) {\rm d}\zeta}v_0(X(x_0,s_0))\\
 &\hspace{4.7cm}-\int_{s_0}^se^{-2\int_{s_0}^r \bar{a}(\zeta) p'(X(x_0,\zeta )) \dzeta} \bar{a}(r) p''(X(x_0,r))u(X(x_0,r),r) {\rm d}r.
}

% \red{I would change $z$ for something more time related, $z$ is a bit confusing here as we use it later as a space , maybe $t$ like before?Also, one day I will change the order of variables in this paper, so that the time comes before space, for the moment we have 50/50} 
% \blue{changed $z$ to $r$.}

% \blue{ For solutions of ODEs, like $X$, time usually comes first. With PDEs, I usually go with $(x,t)$.   We can go either way.}

Since function $a(t)$ has a limit and it is non-zero, the same applies to function $\bar{a}(s)$ when $s\to+\infty.$ Thus if $s_0$ is large enough, $\bar{a}(s)$ has a definite sign. We will only consider the case when $\lim_{s\to+\infty}\bar{a}(s)$ is positive. The other case is considered similarly and will be omitted.
We can estimate:
\eq{
&\left|\int_{s_0}^se^{{-}2\int_{s_0}^r \bar{a}(\zeta) p'(X(x_0,\zeta )) \dzeta} \bar{a}(r) p''(X(x_0,r))u(X(x_0,r),r) \dr \right|\\
&\leq
\max_{r\in[s,s_0]}|u(X(x_0,r),r)|\max_x\frac{|p''(x)|}{p'(x)} \int_{s_0}^se^{2\int_{s_0}^r {|\bar{a}(\zeta)| }p'(X(x_0,\zeta )) \dzeta} |\bar{a}(r)|p'(X(x_0,r))\,\dr\\
&=
\max_{r\in[s,s_0]}|u(X(x_0,r), r)|\max_x\frac{|p''(x)|}{p'(x)} \int_{s_0}^se^{2\int_{s_0}^r {|\bar{a}(\zeta)|} p'(X(x_0,\zeta )) \dzeta} \bar{a}(r)p'(X(x_0,r))\,dr\\
&={}
\frac{1}{2} \max_{r\in[s,s_0]}|u(X(x_0,r),r)|\max_x\frac{|p''(x)|}{p'(x)}( e^{2\int_{s_0}^{{s}} {|\bar{a}(\zeta)|} p'(X(x_0,\zeta )) \dzeta}-1)\\
&\leq
\max_{r\in[s,s_0]}|u(X(x_0,r),r)|\max_x\frac{|p''(x)|}{p'(x)} e^{2\int_{s_0}^{\Red{s}}  {|\bar{a}(\zeta)|} p'(X(x_0,\zeta )) \dzeta},
}
where we used that $p'(x), p'(X(x_0,\zeta ))\geq0$ for $\zeta\in[s_0,+\infty)$. 

From this and \eqref{est:u}, we get 
\eq{
\label{est:v}
|v(X(x_0,s),s)|{}&\leq{} \left(|{v_0}(X(x_0,s_0))|{}+{}\max_{r\in[s,s_0]}|u(X(x_0,r),r)|\max_x\frac{|p''(x)|}{|p'(x)|}\right)e^{2\int_{s_0}^s {|\bar{a}(\zeta)|} p'(X(x_0,\zeta )) \dzeta}\\
&\leq 
C_0\left(1+\max_x\frac{|p''(x)|}{|p'(x)|}\right)e^{3\int_{s_0}^s {|\bar{a}(\zeta)|}p'(X(x_0,\zeta))\,\dzeta}\\
&\leq C_0(1+c_p)e^{3\int_{s_0}^s {|\bar{a}(\zeta)|}p'(X(x_0,\zeta))\,\dzeta}.
}
We now claim for all $s,s_0$ ($s\geq s_0$) and $x_0,$
\begin{equation}
\label{uniform p' integral}
\int_{s_0}^s {|\bar{a}(\zeta)|}p'(X(x_0,\zeta))\,\dzeta=\int_{s_0}^s\bar{a}(\zeta) p'(X(x_0,\zeta))\,{\rm d}\zeta \leq -\log p_{min}.
\end{equation}
To see this, let us change the variable in the last integral $z=X(x_0,\zeta).$ 
Since  $\bar{a}(\zeta)\geq0,$ on $\zeta\in[s_0,+\infty),$
we have $X(x_0,s)\geq X(x_0,s_0),$ when $s>s_0,$ and so we get
\begin{multline*}
\int_{s_0}^s\bar{a}(\zeta)p'(X(x_0,\zeta))\,\dzeta
\\
{}={}
\int_{X(x_0,s_0)}^{X(x_0,s)}\frac{p'(z)}{p(z)}\,\dz
{}
={}
\log p(X(x_0,s))-\log p(X(x_0,s_0))
\leq -\log p(X(x_0,s_0))
{}\leq{} -\log p_{min}.
\end{multline*}

% If $\bar{a}(\zeta)\leq0,$ 
% \red{what happens if $\bar{a}(\zeta)$ changes sign?}
% \blue{Explanations added after equation (135) and at the beginning of the proof.
% I've commented out the case when $\bar{a}$ is negative. It's completely similar}

%for $\zeta\in[s_0,+\infty),$
%then $X(x_0,s)\leq X(x_0,s_0),$ and we get
%\begin{multline*}
%\int_{s_0}^s|\bar{a}(\zeta)|p'(X(x_0,\zeta))\,\dzeta
%{}={}
%-\int_{X(x_0,s_0)}^{X(x_0,s)}\frac{p'(z)}{p(z)}\,dz
%\\
%{}
%=
%{}
%-\log p(X(x_0,s))+\log p(X(x_0,s_0))
%\leq -\log p(X(x_0,s))\leq -\log p_{min}.
%\end{multline*}
Using \eqref{uniform p' integral} in \eqref{est:u} and \eqref{est:v},
 we conclude that  
 \begin{equation*}
\sup_{s\geq s_0,\,x\in\mathbb{R}}|u(x,s)|,\, \sup_{s\geq s_0,\,x\in\mathbb{R}}|v(x,s)|{}\leq{} \frac{\hat{c}_p}{2p_{min}^3},\quad \hat{c}_p:=2C_0(1+c_p).
 \end{equation*}
Thus, equation  \eqref{eq:exp moment h} is now established.

\medskip

\noindent{\emph{Step 5.}} Derivation of \eqref{x<x_0} for the rescaled solution $h$.

From \eqref{eq:exp moment h}, 
we conclude that $\intO{h(x,s)\phi(x,s)}$ is differentiable in $s$ and its derivative can be estimated as
\begin{equation*}
\frac{{\rm d}}{\ds}\intO{h(x,s)\phi(x,s)}{}\leq{}
\frac{\hat{c}_p}{p_{min}^3}\bar{d}(s)\intO{h(x,s)\phi(x,s)}.
\end{equation*}
This implies that 
\begin{equation}
\label{est:moment_h_2}
\intO{h(x,s)\phi(x,s)}{}\leq{}
e^{\frac{\hat{c}_p}{p_{min}^3}\int_{s_0}^s \bar{d}(\zeta)\,{\rm d}\zeta}\intO{h(x,s_0)\phi(x,s_0)}{}={}C_{s_0}e^{\frac{\hat{c}_p}{p_{min}^3}\int_{s_0}^s \bar{d}(\zeta)\,{\rm d}\zeta}.
\end{equation}
From \eqref{X:ODE} we obtain
\[
X(x_0,s)-x_0\geq p_{min}\int_{s_0}^s\bar{a}(\zeta)\,\dzeta,
\]
from which we conclude that
$ x-Y(x,s)\geq p_{min}\int_{s_0}^s\bar{a}(\zeta)\,\dzeta$ for any $x.$

Then, for $x<-2,$ using the definition of  $\psi$ and the fact that $Y(x,s)<x,$  we get
\begin{equation*}
\phi(x,s){}={}\psi(Y(x,s))
{}={}e^{-Y(x,s)}
{}={}e^{-Y(x,s)+x -x}
{}\geq{}
e^{-Y(x,s)+x}
{}\geq{}e^{p_{min}\int_{s_0}^s\bar{a}(\zeta)\,\dzeta}.
\end{equation*}
Using this in \eqref{est:moment_h_2}, we get
\begin{equation}
\label{int h to zero}
\int_{x<-2}h(x,s)\,\dx{}\leq{}
C_{s_0}e^{-p_{min}\int_{s_0}^s\bar{a}(\zeta)\,\dzeta +
\frac{\hat{c}_p}{p_{min}^3}\int_{s_0}^s \bar{d}(\zeta)\,\dzeta}.
\end{equation}
Using the definition of $\bar{d},$ \eqref{def:bar d}, the bounds $\bar{a}^2(\zeta)\leq \bar{a}(\zeta),$ 
 $|\bar{b}(\zeta)|\leq 1,$ and 
selecting {$\tau$ small enough so that}
\[
\sqrt{\tau}(M-1)< \frac{p_{min}^4}{\hat{c}_p},
\]
we get
\[
-p_{min}\int_{s_0}^s\bar{a}(\zeta)\,\dzeta +
\frac{\hat{c}_p}{p_{min}^3}\int_{s_0}^s \bar{d}(\zeta)\,\dzeta
{}\leq{}
-\frac{p_{min}}{2}\int_{s_0}^s\bar{a}(\zeta)\,\dzeta+\frac{\hat{c}_p}{2p_{min}^3}\sqrt{\tau}(s-s_0).
\]
Finally, we choose $s_0$  sufficiently large so that 
\[
\inf_{[s_0,+\infty]}\bar{a}(s)> \frac{\hat{c}_p}{p_{min}^4}\sqrt{\tau},
\]
then, as $s\to+\infty$, we have
\[
-\frac{p_{min}}{2}\int_{s_0}^s\bar{a}(\zeta)\,\dzeta+\frac{\hat{c}_p}{2p_{min}^3}\sqrt{\tau}(s-s_0)\leq
\left(-\frac{p_{min}}{2}\inf_{[s_0,+\infty]}\bar{a}(s)+
\frac{\hat{c}_p}{2p_{min}^3}\sqrt{\tau}\right)(s-s_0)
\to -\infty,
\]
and the integral on the left-hand side of \eqref{int h to zero} converges to zero.
Thus, we obtain \eqref{x<x_0} with $x_0=-2.$ This leads to a contradiction, as was noted above.

\medskip 

\noindent{\emph{Step 6.}} Derivation of the second inequality in \eqref{x<x_0}.\\
Let be $s_0$ is chosen sufficiently large so that 
\[
\sup_{[s_0,+\infty]}\bar{a}(s)<- \frac{\hat{c}_p}{p_{min}^4}\sqrt{\tau}.
\]

From the ODE for $X(x_0,s)$ we obtain
\[
X(x_0,s)-x_0\leq p_{min}\int_{s_0}^s\bar{a}(s)\,\ds,
\]
from which we conclude that
$ x-Y(x,s)\leq p_{min}\int_{s_0}^s\bar{a}(s)\,\ds$ for any $x.$

Then, for $x>2,$ using the definition of  $\psi$ and the fact that $Y(x,s)>x,$  we get
\begin{equation}
\phi(x,s){}={}\psi(Y(x,s))
{}={}e^{Y(x,s)}
{}={}e^{Y(x,s)-x +x}
{}\geq{}
e^{Y(x,s)-x}
{}\geq{}e^{-p_{min}\int_{s_0}^s\bar{a}(s)\,\ds}.
\end{equation}

Using this in \eqref{est:moment_h_2} we get
\begin{equation}
\int_{x>2}h(x,s)\,\dx{}\leq{}
C_{s_0}e^{ p_{min}\int_{s_0}^s\bar{a}(s)\, \ds+
\frac{\hat{c}_p}{p_{min}^3}\int_{s_0}^s\bar{d}(s)\,\ds}\to0,\quad s\to+\infty.
\end{equation}
This implies the second inequality in \eqref{x<x_0} with $x_0=2,$ and we arrive at a contradiction, as was noted above.

From \eqref{est:lim_a} and the fact that $\kappa {}={}\frac{M_c-1}{M-1}\in
(\frac{M_c-1}{M},\frac{M_c}{M}),
$
it follows that, if $\tau$ is sufficiently small, then the condition for the aggregate learning, \eqref{def:learning f}, is satisfied.
\end{proof}

\begin{remark}
    We can re-write equation \eqref{eq:f20}, using \eqref{eq:h_tau}, in the parameters $(h,\tau,M)$ characterizing the original learning model as:
    \[
\partial_t f + \frac{h(M-1)}{\tau}a(t)\partial_x(pf) - \frac{h^2(M-1)^2}{2\tau}\left(a^2(t)+\frac{1}{M-1}b(t)\right)\partial^2_{xx}(pf){}={}0,\quad (x,t)\in\mathbb{R}\times\mathbb{R}^+.
\]
    The coefficient $h(M-1)/\tau$ in the transport term  can be identified as a time scale for aggregate learning, i.e. the time scale for $a(t)$ to diminish. Then, assuming that aggregate learning has already happened, i.e. $a(t)=0,$  we informally assign the time of scale of sorting as the diffusion coefficient $h^2(M-1)/(2\tau).$ 
\end{remark}

% \section{Discussion of scope and limitations}
% \Red{Highlights of this paper: maybe something about that it captures two empirically observed phenomena in one PDE framework: aggregate learning + sorting (not just equilibrium convergence) Often ODE models are not mathematically tractable to obtain results like this, and only the PDE framework can be used to verify formation of some pattern. But because I am massively out of depth here, I don't want to write something stupid.}
% \Red{Maybe discuss this Fokker–Planck truncation approximation, and molecular chaos assumption? Is there any numerical study possible to cite here to verify our observations about sorting? This would make the paper stronger for applied journal, but if we want to stay with not so applied, we can just say some ideas for upcoming work maybe?}

\bigskip

\noindent{\bf{Acknowledgement.}}  
The work of E.Z. was  supported by the EPSRC Early Career Fellowship no. EP/V000586/1.

\vspace{4mm}
\noindent{\bf{Data Availability.}} Data sharing is not applicable to this article as no datasets were generated or analyzed during the current study.
\vspace{4mm}

\noindent{\bf{Conflicts of interest.}} All authors certify that there are no conflicts of interest for this work.

\bigskip

\noindent{\bf{Publishing licence.}} For the purpose of open access, the author has applied a Creative Commons Attribution (CC BY) licence to any Author Accepted Manuscript version arising from this submission.
\bibliographystyle{plain}
\bibliography{references}

@article{ROTH1995164,
title = {Learning in extensive-form games: Experimental data and simple dynamic models in the intermediate term},
journal = {Games and Economic Behavior},
volume = {8},
number = {1},
pages = {164-212},
year = {1995},
issn = {0899-8256},
doi = {https://doi.org/10.1016/S0899-8256(05)80020-X},
url = {https://www.sciencedirect.com/science/article/pii/S089982560580020X},
author = {Alvin E. Roth and Ido Erev}
}

@article{benaim1999,
  title = {Dynamics of stochastic algorithms},
  author = {Bena\"{i}m, Michael},
  year = {1999},
  journal = {S\'{e}minaire de probabilit\'{e}s de Strasbourg},
  volume = {33},
  number = {},
  pages = {1-68}
}

@article{Arthur1994,
  title = {Complexity in Economic Theory. Inductive Reasoning
and Bounded Rationality},
  author = {Arthur, W. Brian},
  year = {1994},
  journal = {American Economic Review},
  volume = {82},
  number = {2},
  pages = {406--411}
}

@article{Carrillo,
  author  = {Jos{\'e} Antonio Carrillo and Robert Joseph McCann and C{\'e}dric Villani},
  title   = {Kinetic equilibration rates for granular media and related equations: Entropy dissipation and mass transportation estimates},
  journal = {Revista Matem{\'a}tica Iberoamericana},
  volume  = {19},
  number  = {3},
  pages   = {971--1018},
  year    = {2003}
}

@article{Otto2001,
  author  = {Felix Otto},
  title   = {The geometry of dissipative evolution equations: The porous medium equation},
  journal = {Communications in Partial Differential Equations},
  volume  = {26},
  number  = {1-2},
  pages   = {101--174},
  year    = {2001}
}

@article{CaChZa,
  author  = {Jos{\'e} Antonio Carrillo and Young-Pil Choi and Ewelina Zatorska},
  title   = {On the pressureless damped Euler--Poisson equations with non-local forces: Critical thresholds and large-time behavior},
  journal = {Mathematical Models and Methods in Applied Sciences},
  volume  = {26},
  number  = {12},
  pages   = {2311--2340},
  year    = {2016}
}

@article{BaDeZa,
  author  = {Julien Barr{\'e} and Pierre Degond and Ewelina Zatorska},
  title   = {Kinetic theory of particle interactions mediated by dynamical networks},
  journal = {Multiscale Modeling \& Simulation},
  volume  = {15},
  number  = {3},
  pages   = {1294--1323},
  year    = {2017}
}

@incollection{Meleard,
author = {M{'e}l{'e}ard, Sylvie},
title = {Asymptotic behaviour of some interacting particle systems; {McKean--Vlasov} and {Boltzmann} models},
booktitle = {Probabilistic Models for Nonlinear Partial Differential Equations},
editor = {Talay, Denis and Tubaro, Luciano},
series = {Lecture Notes in Mathematics},
volume = {1627},
pages = {42--95},
year = {1996},
publisher = {Springer},
address = {Berlin, Heidelberg},
doi = {10.1007/BFb0093175}
}

@article{ChaintronDiez,
author = {Chaintron, Louis-Pierre and Diez, Antoine},
title = {Propagation of chaos: A review of models, methods and applications. I. Models and methods},
journal = {Kinetic and Related Models},
volume = {15},
number = {6},
pages = {895--1015},
year = {2022},
doi = {10.3934/krm.2022017}
}

@article{Dobrushin,
author = {Dobrushin, Roland L.},
title = {Vlasov equations},
journal = {Functional Analysis and Its Applications},
volume = {13},
number = {2},
pages = {115--123},
year = {1979},
doi = {10.1007/BF01077243}
}

@incollection{McKean,
author = {McKean, Henry P. Jr.},
title = {Propagation of chaos for a class of nonlinear parabolic equations},
booktitle = {Stochastic Differential Equations},
series = {Lecture Series in Differential Equations},
volume = {7},
pages = {41--57},
year = {1967},
publisher = {Catholic University},
address = {Washington, D.C.}
}

@incollection{Sznitman,
author = {Sznitman, Alain-Sol},
title = {Topics in propagation of chaos},
booktitle = {{'E}cole d'{'E}t{'e} de Probabilit{'e}s de Saint-Flour XIX---1989},
editor = {Hennequin, Paul-Louis},
series = {Lecture Notes in Mathematics},
volume = {1464},
pages = {165--251},
year = {1991},
publisher = {Springer},
address = {Berlin, Heidelberg},
doi = {10.1007/BFb0085169}
}

@article{duffy2005,
  title = {Learning, information, and sorting in market entry games: theory and evidence},
  author = {Duffy, John. and Hopkins, Ed},
  year = {2005},
  journal = {Games and Economic Behavior},
  volume = {51},
  number = {},
  pages = {31-62}
}

@article{erev1998,
  title = {Coordination, ``magic'', and reinforcement learning in a market entry game},
  author = {Erev, Ido and Rapoport, Anatole},
  year = {1998},
  journal = {Games and Economic Behavior},
  volume = {23},
  number = {},
  pages = {46-175}
}

@article{erev1998b,
  title = {Predicting how people play games: reinforcement learning in experimental games with unique, mixed strategy equilibrium},
  author = {Erev, Ido and Roth, Alvin E},
  year = {1998},
  journal = {American Economic Review},
  volume = {88},
  number = {},
  pages = {848-881}
}

@book{fudenberg1998,
  title     = {The Theory of Learning in Games},
  author    = {Fudenberg, Drew and Levine, David K },
  year      = {1998},
  publisher = {MIT Press},
  address   = {Cambridge, MA}
}

@book{risken1996,
  title     = {The Fokker-Planck Eqution},
  author    = {Risken, Hannes},
  year      = {1996},
  publisher = {Springer},
}

@article{whitehead2008,
  title = {The El Farol bar problem revisited: reinforcement learning in a potential game},
  author = {Whitehead,
  Duncan},
  year = {2008},
  journal = {Edinburgh School of Economics, Discussion Paper Series},
  volume = {186},
  number = {},
  pages = {848-881},
  publisher = {}
}

@book{Krylov,
  title     = {Lectures on Elliptic and Parabolic Equations in Sobolev Spaces},
  author    = {Krylov, Nikolay V.},
  year      = {2008},
  publisher = {AMS Graduate Studies in Mathematics, Vol. 96}
}

@article{Perepelitsa2021,
title = {A model of cultural evolution in the context of strategic conflict},
journal = {Kinetic and Related Models},
volume = {14},
number = {3},
pages = {523-539},
year = {2021},
issn = {1937-5093},
doi = {10.3934/krm.2021014},
url = {https://www.aimsciences.org/article/id/dfc51c77-cab9-4525-815f-ec61cead6173},
author = {Misha Perepelitsa}
}

@article{Perepelitsa2019,
title = {Adaptive learning in large populations},
journal = {Journal of Mathematical Biology},
volume = {79},
number = {},
pages = {2237-2253},
year = {2019},
author = {Misha Perepelitsa},
}

@article{Degond2014,
title = {Large-Scale Dynamics of Mean-Field Games Driven by Local Nash Equilibria},
journal = {Journal Nonlinear Sciences},
volume = {24},
number = {},
pages = {93-115},
year = {2014},
author = {Degond, Pierre and Liu, Jian-Guo. and Ringhofer, Christian},
}

@article{frouvelle2012,
author = {Frouvelle, Amic and Liu, Jian-Guo},
title = {Dynamics in a Kinetic Model of Oriented Particles with Phase Transition},
journal = {SIAM Journal on Mathematical Analysis},
volume = {44},
number = {2},
pages = {791-826},
year = {2012},
doi = {10.1137/110823912},
URL = {        https://doi.org/10.1137/110823912},
eprint = {         https://doi.org/10.1137/110823912}
}

@article{constantin2004,
title = "Asymptotic States of a Smoluchowski Equation",
author = "Peter Constantin and Ioannis Kevrekidis and Edris Titi",
year = "2004",
volume = "174",
pages = "365--384",
doi = "https://doi.org/10.1007/s00205-004-0331-8",
journal = "Archive for Rational Mechanics and Analysis"
}

@article{HARLEY1981,
title = {Learning the evolutionarily stable strategy},
journal = {Journal of Theoretical Biology},
volume = {89},
number = {4},
pages = {611-633},
year = {1981},
issn = {0022-5193},
doi = {https://doi.org/10.1016/0022-5193(81)90032-1},
url = {https://www.sciencedirect.com/science/article/pii/0022519381900321},
author = {Calvin B. Harley}
}

\end{document}